\documentclass[12pt]{ociamthesis}
\usepackage{amssymb}
\usepackage{amsmath}
\usepackage{graphicx}
\usepackage[colorlinks=true, allcolors=blue]{hyperref}
\usepackage{amsfonts}
\usepackage{bbm}
\usepackage{amsthm}
\usepackage[utf8]{inputenc}
\usepackage{enumerate}
\newcommand{\leb}[1]{\text{Leb} \left( #1 \right)}

\newtheorem{theorem}{Theorem}[section]

\newtheorem{lemma}[theorem]{Lemma}
\newtheorem{proposition}[theorem]{Proposition}
\newtheorem{conjecture}[theorem]{Conjecture}

\newtheorem{definition}[theorem]{Definition}

\title{Brownian Bees with Drift:\\[1ex]     
        Finding the Criticality}   

\degree{MMath Mathematics}     
\degreedate{Trinity 2023}         

\begin{document}


\setcounter{secnumdepth}{3}
\setcounter{tocdepth}{3}

\maketitle 

\begin{center}
     {\large \textbf{Abstract}}
\end{center}

{\large
This dissertation examines the impact of a drift $\mu$ on Brownian Bees, which is a type of branching Brownian motion that retains only the $N$ closest particles to the origin. The selection effect in the $0$-drift system ensures that it remains recurrent and close to the origin. The study presents two novel findings that establish a threshold for $\mu$: below this value, the system remains recurrent, and above it, the system becomes transient. Moreover, the paper proves convergence to a unique invariant distribution for the small drift case. The research also explores N-BBM, a variant of branching Brownian motion where the $N$ leftmost particles are retained, and presents one new result and further discussion on this topic.}          

\begin{romanpages}          
\tableofcontents            

\end{romanpages}            

\chapter{Introduction}

Branching Brownian motion (BBM) is a stochastic process involving independent particles moving according to Brownian motion. Initially there is a single particle following a Brownian trajectory, until after an independent exponential time of rate 1, it branches into two new Brownian particles, each of which has its own exponential branching clock of rate 1. The study of BBM in its most basic form dates back to 1975; introduced by McKean \cite{McKean_FKPP}.

As is common with many mathematical models, after the introduction of BBM different authors considered making small, physically motivated changes to the construction and considering what effects this has on its properties. One such alteration is the "N-BBM"; a model to which a large part of this dissertation is focused on. In this model we start with $N$ particles performing branching Brownian motion in one dimension, and whenever a particle branches, such that there are $N+1$ particles, we instantaneously remove the particle furthest to the left, i.e. the particle whose position is less than any other particles (and with arbitrary choice in a tie).

Another alteration is the "Brownian Bees" process\footnote{The name "Brownian Bees" was initially suggested by Jeremy Quastel \cite{swarm_limit} and it due to a rough analogy of the Brownian motions looking like a "swarm of bees".}. Evolving similarly to N-BBM, such that whenever we have more than $N$ particles one is removed - however this time we remove the particle furthest away from the origin. This process is a more well behaved object than the N-BBM due to the compacting effect of the selection rule, leading to the existence of limits for the cases $N \rightarrow \infty$ and $t \rightarrow \infty$. A pair of papers by J. Berestycki et al established a nice result of how in a certain sense these limits commute. \cite{free_boundary_problem, swarm_limit}.

In this dissertation we consider the effect of adding a drift $\mu$ to the Brownian Bee process, so that each particle evolves according to a Brownian motion with drift. The results of this paper can then be summarised as: There is a critical value $v_N$ such that if $|\mu| > v_N$ then the system is transient, and if $|\mu| < v_N$ the system is recurrent, and furthermore this $v_N$ is the asymptotic speed of the N-BBM. The precise statement of these will be given by Theorems \ref{Brownian_Bees_supercritical} and \ref{subcritical_stationarity}.

We can construct a generalisation of the above processes (following \cite{brunet-derrida}\footnote{In this paper N.Berestycki and Zhao are exploring multidimensional Brunet-Derrida systems, however for this dissertation we shall stick primarily to the one dimensional process.}), where we describe the selection rule by a "score" function $s: \mathbb{R}^d \rightarrow \mathbb{R}$, and then we refer to the process as a "Brunet Derrida system with score function $s$". Furthermore we may add a drift $\mu \in \mathbb{R}^d$ to the Brownian motions driving the particles between branching times. The process $Y^{(N, \mu)} = \left(Y_1(t), \ldots Y_n(t)\right)$ is then defined informally by:
\begin{itemize}
    \item Each particle moves according to an independent Brownian motion with drift $\mu$.
    \item Every time an exponential clock of rate $N$ fires, the particle with the lowest value of $s(Y_i(t))$ jumps to the position of a randomly chosen particle $Y_k(t)$ where $k$ is chosen uniformly from $\{1, \ldots, N \}$
\end{itemize}

Note in this definition that a particle can "jump to itself", where in this case no change will occur.

Following this notation, the N-BBM is recovered by the function $s(x) = x$, whereas the Brownian Bee process is recovered by $s(x) = -|x|$.

These models have received some attention in the literature since the initial conception of a one-dimensional such model, by Brunet, Derrida, Mueller and Munier 2006/7 in the papers \cite{initial_brunet_derrida_1, initial_brunet_derrida_2}.

\begin{section}{Motivation}
Some physical motivation for the study of Brunet-Derrida systems can be provided by the study of evolution \cite{model_ancestry, model_genealogy}
\end{section}. Broadly speaking, the Brownian particles represent the fitness of different members of an asexually reproducing species. The population is taken to be fixed at size $N$ due to external factors, and the members of the species that survive are those with the highest fitness function $s$. The Brownian motion aspect then represents some white noise caused by mutations upon reproduction. The different functions of $s$ then allow for modelling of different types of evolutionary behaviour.

We can give additional motivation for the model of Brownian Motion with drift, although it should be prefaced with the fact that this author is not a biologist. In the standard Brownian Bees model, it could be viewed that the origin represents the "optimum fitness" for the environment that the species is in - giving the interpretation that the convergence in $t$ of the Brownian Bee model represents the species eventually adapting to the environment they are in (\cite{swarm_limit} Theorem 1.2). Adding in a drift to the Brownian motion can be viewed as moving the origin to position $\mu t$ at time $t$, this could represent this "optimum fitness" changing over time, for example due to climate change or other external factors. It is then useful to study whether the population is able to adapt fast enough to keep up with these external changes.

\begin{section}{Notation and Formal Construction}

In this section we will formally construct processes in enough generality to cover N-BBM, Brownian Bees and Brownian Bees with drift. We shall refer to such systems as "Brunet-Derrida systems with drift $\mu$"

We will construct a process with a drift $\mu \in \mathbb{R}^d$ and a "score function" $s: \mathbb{R}^d \rightarrow \mathbb{R}$, where the score function determines which particle is removed at branching times. For example for standard N-BBM in one dimension, we take $\mu = 0$, and $s(x) = x$.

We shall make the convention that in $\mathbb{R}^d$, $d>1$ we label the particles such that
$$s(Y_1(t)) \leq \ldots \leq s(Y_N(t))$$
however in $\mathbb{R}$ it is more convenient to use the ordering 
$$Y_1(t) \leq \ldots \leq Y_N(t)$$

We generally label the process by $Y^{(N, \mu)}$ (where the score function used will be clear from context). Furthermore, we write $Y^{(N, \mu)} = (Y_1, \ldots Y_N)$, suppressing the dependence on $N$ and $\mu$, which again will be clear from context.

Now for the formal construction, we adapt from \cite{brunet-derrida}:

Let $\left(J_i\right)_{i \geq 0}$ be the jump times of a Poisson process with rate $N$ with $J_0=0$, and let $\left(K_i\right)_{i \geq 1}$ be an independent sequence of i.i.d. uniform random variables on $\{1, \ldots, N\}$. The process is started in some given initial condition ($Y_1(0), \ldots Y_N(0)$). Then inductively, for each $i \geq 1$, assuming that the system is defined up to time $J_{i-1}$ with $Y_1\left(J_{i-1}\right) \geq \cdots \geq Y_N\left(J_{i-1}\right)$, we define

$$ Y_n(t)=Y_n\left(J_{i-1}\right)+B^n(t) + \mu t, \quad t \in\left[J_{i-1}, J_i\right)$$

, where $\left(B^n(t), t \geq 0\right)_{n = 1}^N$ are independent Brownian motions in $\mathbb{R}$, independent from $\left(K_i\right)$ and $\left(J_i\right)$. At time $J_i$, we duplicate particle $Y_{K_i}\left(J_i^{-}\right)$ and remove the particle $\min _{1 \leq n \leq N} s\left(Y_n\left(J_i^{-}\right)\right)$. Note that if the duplicated particle is the particle with minimum score, the net effect is that nothing happens. We now relabel the particles over this interval in so that they are increasing in $\mathbb{R}$ (or in the case $\mathbb{R}^d$, $d>1$ we order such that they are increasing in their score function)
$$
Y_1(t) \leq \ldots \leq Y_N(t), \quad t \in\left[J_{i-1}, J_i\right]
$$

The process $Y^{(N, \mu)} = (Y_1(t), \ldots Y_N(t))$ is then a Brunet-Derrida system with score function $s$ and drift $\mu$.

We can now recover the processes mentioned at the beginning by taking:

\begin{itemize}
    \item $d=1$, $s(x) =x$ and $\mu = 0$ for standard N-BBM
    \item $s(x) = -||x||$, $\mu=\mu$ and $d=d$ for $d$-dimensional Brownian bees with drift $\mu$
\end{itemize}

Additionally, we introduce the concept of \emph{ancestors} and \emph{children}. The time $s$ ancestor of a particle $Y_n(t)$ for $s < t$ is a particle that $Y_n(t)$ is a direct descendent of through a duplication event. Additionally, the time $t$ children of a particle $Y_k(s)$ are all the particles that are directly descended from $Y_k(s)$ through duplication events.
We can see that each particle has precisely one ancestor at any given previous time, however a particle may have between $0$ and $N$ children at any later time.

We further remark that since the particles are ordered by position at any time $t$ it is \emph{not} necessarily the case that $Y_n(s)$ is an ancestor of $Y_n(t)$.

We will also refer to $\mathcal{F}_t$, the filtration generated by the whole process; so if $N_t := \# \{i \geq 1 : J_i \leq t \}$ is the counting process associated to the jump times $J_i$ then,

$$\mathcal{F}_t := \sigma \left( (B^1_s)_{s \leq t}, \ldots, (B^N_s)_{s \leq t}, (N_s)_{s \leq t}, (K_i)_{i \leq N_t}\right)$$

\end{section}

\chapter{N-BBM}
\begin{section}{Heuristics and Current Results}

One-dimensional Branching Bronwian motion with selection (N-BBM) first received attention, in Brunet, Derrida, Mueller and Munier's influential papers \cite{ initial_brunet_derrida_1, initial_brunet_derrida_2}. The original purpose of such a model was to study "noisy F-KPP equations" that is the \emph{Fisher–Kolmogorov–Petrovskii–Piskounov equation}
\begin{equation}
\label{F-KPP}
\frac{\partial}{\partial t} u = \frac{1}{2} \left(\frac{\partial}{\partial x} \right)^2 u + F(u)
\end{equation}

Where the term $F(u) = \beta (1 - u - f(1-u))$.

Since the models conception, however there have been several other applications of its study. For a more detailed discussion of these see the introduction of Pascal Malliads 2012 thesis \cite{BBM_with_selection}.

A quick relation of the N-BBM to the F-KPP equation is from the free branching Brownian motion process, that is a particle started at the origin and moving according to a Brownian motion. The particle branches to form new independent branching particles at rate 1, with no killing occurring.

It can then be shown that the probability $u(x, t)$, that there exists a particle to the right of $x$ at time $t$ satisfies Equation \ref{F-KPP} (In fact this is the very reason BBM was first introduced \cite{McKean_FKPP}).

The precise reason for the relevance of the N-BBM model is a little more subtle, and comes from the study of the "cut-off" equation introduced by Brunet and Derrida in \cite{brunet_derrida_original}, where the forcing term in the F-KPP equation \ref{F-KPP} is multiplied by $\mathbbm{1}_{u \geq N^{-1}}$, allowing for better modelling of real world phenomenon. For the purposes of this dissertation we simply assure the reader that N-BBM is a worthy object of study and direct them to \cite{BBM_with_selection} for more details on this.

Letting $X^{(N)} = (X_1(t), \ldots X_N(t))$ be a standard N-BBM; some key conjectures made by Brunet and Derrida about this model in \cite{BD_Conjectures} are amongst others:

\begin{enumerate}
    \item For every $N$, $\lim_{t \rightarrow \infty} \frac{X_1(t)}{t} =  \lim_{t \rightarrow \infty} \frac{X_N(t)}{t} = v_N$
    as $N \rightarrow \infty$ 
    \item $v_N = \sqrt{2} - \frac{\pi^2}{\sqrt{2}(\log N)^2}\left(1 - \frac{6 \log \log N}{(\log(N)^3}(1 + o(1)) \right)$
    \item $\text{diam}_t := X_N(t) - X_1(t) \approx  \frac{\log(N)}{\sqrt{2}}$
    \item The genealogical timescale of the population is $(\log N)^3$ and converges to the Bolthausen–Sznitman coalescent
\end{enumerate}

We won't delve into the $4^{\text{th}}$ point here too much, other than remarking that it is a very deep result and more discussion and definitions are given in \cite{genealogical_time_scale}. Broadly speaking, understanding the genealogical time scale allows us to understand how far back in time we need to go before we find a common ancestor of the whole living population. It is easy to see then why this is a point of interest, due to the study of N-BBM being partially motivated by population genetics. 

As for the other points on this list, A proof of (1) was first given for a very similar model by \cite{Brunet-Derrida_velocity} by Bérard and Gouéré, and this was adapted by N.Berestycki and Zhao for the case of N-BBM. The authors extended this result to $d$ dimensions (Theorem 1.1 in \cite{brunet-derrida}), however we state their result for the simplest case $d=1$ here:

\begin{theorem}
\label{speed_linear_score}
Let $N>1$ and let $(X_1, \ldots X_N)$ be an N-BBM process. Then
\begin{equation}
    \frac{|X_N(t) - X_1(t)|}{t} \rightarrow 0
\end{equation}
as $t \rightarrow \infty$ almost surely.

Moreover,
\begin{equation}
    \frac{X_1(t)}{t} \rightarrow v_N
\end{equation}
almost surely, where $v_N$ are a deterministic constants
\end{theorem}
This result then tell us that for large times the N-BBM moves in a ballistic manor, and furthermore has a deterministic asymptotic velocity $v_N$.

Point (2) is harder to approach, though it has been partially settled by Bérard and Gouéré in the same paper. They showed that for a branching random walk with selection (a very similar but not quite identical process to N-BBM), the velocity of the system has a correction term of order $(\log N)^{-2} + o\left((\log N)^{-2} \right)$. This result should in theory be easily adaptable to the case of N-BBM, however due to the technical nature of their proof no rigorous adaption to N-BBM has been given. Getting the "third order" $\frac{6 \log \log N}{(\log(N)^3}(1 + o(1))$ correction term out is a harder and still open problem, with currently no known approaches for how to tackle this.

We state this conjecture for the second order correction to $v_N$ here:
\begin{conjecture}
\label{speed_conjecture}
The constants $v_N$ appearing in Theorem \ref{speed_linear_score} have asymptotic expansion in $N$ given by 
\begin{equation}
    v_N = \sqrt{2} - \frac{\pi^2}{\sqrt{2}( \log N)^2} + o\left( (\log N)^{-2}\right)
\end{equation}
\end{conjecture}

Although there is no formal proof for the asymptotic of $v_N$, we do have a proof that $v_N$ are monotonically increasing in $N$, following from a useful coupling of the N-BBM between different values of $N$. This is given as Lemma 2.3 in \cite{brunet-derrida} and we state this here also:
\begin{lemma}
\label{N_coupling}
Let $(X_n(t), 1 \leq n \leq N)_{t \geq 0}$ and $(Y_n(t), 1 \leq n \leq N')_{t \geq 0}$ , $N \leq N'$ be standard N-BBMs. Suppose they are initially ordered such that there is a coupling 
$$Y_1(0) \geq X_1(0); \ldots; Y_N(0) \geq X_N(0) $$
Then we can couple $X(t)$ and $Y(t)$ for all times $t$ such that $X_i(t) \leq Y_i(t)$ for all $t \geq 0$ and $1 \leq i \leq N$
\end{lemma}
A further strong heuristic for why Conjecture \ref{speed_conjecture} should be true is given by \cite{BBM_with_selection}, where it is shown that at the timescale of $(\log N)^3$ the process converges as $N \rightarrow \infty$ to a specified distribution around $v_N t$ (See Theorem 1.1 for exact setup). Where by timescale here, we mean taking the limits $t \rightarrow \infty$ and $N \rightarrow \infty$ simultaneously, such that $t = t_0(\log N)^3$. The speed here is then shown to be as conjectured, however this is not quite enough to settle the result, as we can't unpack the double limiting result as we would wish to.

Returning to the third conjecture of Brunet and Derrida about the diameter, it is not easy to even rigorously define what "$\approx$" means in this context. A related result is proven by \cite{brunet-derrida}, where adapting the result for one dimension, the authors showed that provided the initial conditions satisfy some technical result there exists some constants $a$ and $c$ such that

\begin{equation}
     \liminf_{N \rightarrow \infty} \mathbb{P}\left( \text{diam}_t \leq a \log N\right) = 1, \text{   for } t = c (\log N)^3
\end{equation}

If we allow ourselves to interpret $(\log N)^3$ as some "large" time depending on $N$, this then gives us an upper bound for the diameter.

It is perhaps interesting now to remark on the results of Pain 2015 \cite{L_BBM} who considered a model related to N-BBM. This model, titled "L-BBM" instead of having a fixed population size $N$ has a fixed population diameter $L$, such that if any particle is further than $L$ away from the leading particle, it is instantly killed. The particles otherwise move and branch as in a free branching Brownian motion.

Pain showed that the velocity of the L-BBM has asymptotic speed
\begin{equation}
\label{L_BBM_velocity}
v_L := \sqrt{2} - \frac{\pi^2}{2 \sqrt{2} L} + o\left(\frac{1}{L^2}\right)
\end{equation}
which if we allow ourselves to naively take the heuristic "$L = \frac{\log N}{2}$", exactly recovers the second order conjectured asymptotic of the N-BBM.
Though this gives a nice connection between L-BBM and N-BBM, it is far from rigorous, and in fact the proof of \ref{L_BBM_velocity} does not make use of any coupling to an N-BBM type model.

Despite all this uncertainty surrounding the N-BBM model, some basic properties of the model are unambiguously and easily seen to be true, two of which are:
\begin{itemize}
    \item The process is translation invariant. i.e. if $X = \left(X_1, \ldots X_N \right)$ is an N-BBM started in position $X_1(0), \ldots X_N(0)$, and $Y$ is an N-BBM started in position $Y_k(0) = X_k(0) + \chi$ for some $\chi \in \mathbb{R}$, then we may couple so that $Y_k(t) = X_k(t) + \chi$
    \item The process is a strong Markov process. Since it is driven by a combination of Brownian motions and exponential distributions, both of which are strong Markov processes.
\end{itemize}
\end{section}

\begin{section}{Theorem  \ref{N-BBM_Hitting_Time}}
We now prove a new result (to this author's knowledge) relating to the expectation of the hitting time of an N-BBM when we give the process a drift, $-\mu$.
\begin{theorem}
\label{N-BBM_Hitting_Time}
Let $X = \left(X_1, \ldots X_N\right)$ be a standard N-BBM with killing from the left started at $X_1(0) = \ldots = X_N(0) = 0$. Then let $\mu$ be fixed with $0 \leq \mu < v_N$ and 
$$\tau_{R, \mu} := \inf \{t > 0: X_1(t) - \mu t \geq R\}$$
Then there are constants $\alpha$ and $\beta$ depending on $\mu$ such that
$$\mathbb{E}[\tau_{R, \mu}] < \alpha + \beta R$$
\end{theorem}

\begin{subsection}{Proof of Theorem  \ref{N-BBM_Hitting_Time}}

When giving the proof of this theorem, we break it down into the cases $N =2$ and $N>2$. This is not because the case of $N=2$ is more difficult but rather the opposite; the proof for $N > 2$ does work for $N=2$ however $N=2$ is much simpler. We present this therefore as an easier to follow "warm up" result. Additionally we break up the proof in this manor to record the progress of this dissertation, since before the proof of $N>2$ was found, we coupled $N>2$ to $N=2$ to prove the above result for all $N$, but with the more restrictive assumption of $\mu < v_2 = \frac{3}{8 \sqrt{2}}$.

Instrumental in our proof will be the Lemma:
\begin{lemma}[First passage times of random sums]
\label{random_sum_expectation}
Let $L_1, L_2, \ldots$ be i.i.d random variables with $\mathbb{E}[L_n] > 0$. Then let $S_n := \sum_{k=1}^n L_k$ and for any $R \in \mathbb{R}^+$ let $\tau_R := \inf\{n : S_n \geq R \}$

Then there exists constants $\alpha$ and $\beta$ depending only on the distribution of $L_1$ such that:

$$\mathbb{E}[\tau_R] \leq \alpha + \beta R$$

\end{lemma}

\begin{proof}
This is given by \cite{random_sum_expectations} Equation (1.5). We have also found but omit an elementary proof under the assumption $\mathbb{E}[|L_1|^3] < \infty$. (Using a higher order Chebyshev-like inequality to bound $\mathbb{P}[S_n \leq R]$, and then using $\mathbb{E}[\tau] = \sum \mathbb{P}[\tau \geq n] \leq \sum \mathbb{P}[S_n \leq R]$)

\end{proof}

In our proof we find a way to compare the N-BBM to a random variable that looks like $\sum_{k=1}^n L_k$, which will allow us to use the above Lemma.

\begin{proof}[Proof ($N=2$):]

To prove the case $N=2$ we make use of the fact that the 2-particle system has a regenerative structure that makes computation much easier, since at every time the rightmost particle branches, both particles return to the same point. This allows us to be much more explicit with our calculations, and in fact even allows us to calculate $v_2$ exactly.

We let $T_n$ be the branching times of the leftmost particle of $X$ (since nothing happens when the rightmost particle branches). Then $T_n$ form a Poisson process of rate 1, and we may define a discrete time Markov chain $Z_n := X_1(T_n) - \mu T_n$

Hence if the i.i.d random variables $L_1, L_2, \ldots$ have the distribution of 
$$ L_k \overset{\mathrm{d}}{=}
 \left( \max \left\{ B^1(T), B^2(T) \right\} - \mu T \right)$$
, where $B^1(t)$, $B^2(t)$ are i.i.d Brownian Motions and $T$ is an independent exponential time, then by the strong Markov property at $T_n$ we can see that $Z_n$ has the same distribution as $\sum_{k=1}^n L_k$.
Then by for example \cite{BM_Handbook} 1.0.5 we may find that the distribution for $B^1(T)$ is given by 

$$\mathbb{P}(B^1(T) \in dz) = \frac{1}{\sqrt{2}} e^{-|z| \sqrt{2}}$$

We may use this to calculate $\mathbb{E}[L_k]$. Giving:

$$\mathbb{E}(L_k) = \mathbb{E}\left(\max \left\{ B^1(T), B^2(T) \right\} \right) - \mu \mathbb{E}(T)$$
$$=- \mu + \int_{x=-\infty}^{x=\infty}\int_{y=x}^{y=\infty} \frac{y}{2} e^{-\sqrt{2} (|x| + |y|)}dydx$$
$$= \frac{3}{8 \sqrt{2}} - \mu$$

Where the constant $\frac{3}{8 \sqrt{2}}$ comes from evaluating the integral.\footnote{This constant is in fact equal to $v_2$ since, we can view the process $Z_n$ as a renewal reward process (see e.g. \cite{RossSheldonM.2019Itpm} Chapter 7) and so the Elementary Renewal Theorem for Renewal Reward processes gives that 

$$v_2 - \mu = \lim_{n \rightarrow \infty} \frac{X_1(T_n) - \mu T_n}{T_n} = \frac{\mathbb{E}[L_k]}{\mathbb{E}[T_1]} = \frac{3}{8 \sqrt{2}} - \mu$$}

Then 
$$\mathbb{E}({\tau_{R, \mu}}) = \int_0^{\infty} \mathbb{P}({\tau_{R, \mu}} \geq t) dt =\mathbb{E}\left[ \int_0^{\infty} \mathbbm{1}_{{\tau_{R, \mu}} \geq t} dt\right]$$

by Fubini. If $N_t := \sup \{n : T_n \leq t \}$ is then the counting process corresponding to the Poisson process of branching times $T_n$, then since $T_{N_t} \leq t$ it follows that 
$$\mathbb{E}\left[ \int_0^{\infty} \mathbbm{1}_{{\tau_{R, \mu}} \geq t} dt \right] \leq \mathbb{E}\left[ \int_0^{\infty} \mathbbm{1}_{{\tau_{R, \mu}} \geq T_{N_t}} dt \right] = \mathbb{E}\left[ \sum_{k=1}^\infty (T_{k+1} - T_k)\mathbbm{1}_{{\tau_{R, \mu}} \geq T_{k}}\right]$$
$$ =  \sum_{k=1}^\infty \mathbb{E}\left[(T_{k+1} - T_k)\mathbbm{1}_{{\tau_{R, \mu}} \geq T_{k}}\right] = \sum_{k=1}^\infty  \mathbb{E}\left[\mathbbm{1}_{{\tau_{R, \mu}} \geq T_{k}}\right]$$

Where we have used Fubini again, and that since as ${\tau_{R, \mu}}$ and $T_k$ are stopping times for the strong Markov process $X$, then $T_{k+1} - T_k$ is independent from the $\mathcal{F}_{T_k}$ measurable random variable $\mathbbm{1}_{{\tau_{R, \mu}} \geq T_k}$, and furthermore since $T_{k+1} - T_k$ is a rate 1 exponential distribution, its expectation is $1$. 

Finally then, if we define $\tilde{\tau}_{R, \mu}$ to be the first time that $Z_n$ exceeds $R$, i.e. $\tilde{\tau}_{R, \mu} = \inf \{n: Z_n \geq R \}$, then

$$ \sum_{k=1}^\infty  \mathbb{E}\left[\mathbbm{1}_{{\tau_{R, \mu}} \geq T_{k}}\right] \leq \sum_{k=1}^\infty  \mathbb{P}\left(\tilde{\tau}_{R, \mu} \geq k\right) = \mathbb{E}[\tilde{\tau}_{R, \mu}] \leq \alpha + \beta R$$

Where in the final step we have applied Lemma \ref{random_sum_expectation}, since $Z_n = \sum_{k=1}^n L_k$

\end{proof}

In order to modify this proof so that it works for $N>2$ we have to make some changes. Firstly, we fail to get times $T_n$ such that the process renews at this point, since for $N > 2$ the probability that we have at least three particles in the same position at any time $t >0$ is $0$. The basic idea to get around this issue is to define times 
$$T^{\epsilon}_n \approx \inf \left\{t>T^{\epsilon}_{n-1} + 1: \sup_{1 \leq n,m \leq N} |X_n(t) - X_m(t)| \leq \epsilon \right\}$$

Which are times where the particles are \emph{almost} at the same position. The precise definition of $T_n^{\epsilon}$ is slightly different as we also want to ensure that $T^{\epsilon}_n - T^{\epsilon}_{n-1}$ are i.i.d random variables. We then couple processes $X^{\epsilon, +}$ and $X^{\epsilon, -}$ which correspond to moving all the particles at time $T_n^\epsilon$ to the position of $X_1(T_n^\epsilon)$ and $X_N(T_n^\epsilon)$ respectively. Since the process is very close together at time $T_n^\epsilon$ the result is that we "only lose $\epsilon$" from the drift at each replacement. And then the bounding processes $X^{\epsilon, +}$ and $X^{\epsilon, -}$ can be analysed in a similar manor to the proof of $N=2$, since at the stopping times $T^\epsilon_n$ they are the sum of i.i.d random variables.

\begin{proof}[Proof ($N \geq 2$):]
Let $Y_k(t) := X_k(t) - \mu t$

Then let $\epsilon > 0$ be a constant to be fixed later and iteratively construct processes $Y^{\epsilon, \pm} = (Y^{\epsilon, \pm}_1, \ldots, Y^{\epsilon, \pm}_N)$and the stopping times $T_n^\epsilon$.

Firstly, let $T^\epsilon_0 = 0$ and $Y^{\epsilon, -}_i(0) =0 =Y^{\epsilon, +}_i(0)$ for $1 \leq i \leq N$.
Now, we inductively construct $Y^{\epsilon, -}$ and $Y^{\epsilon, +}$ such that: 
\begin{itemize}
    \item $Y^{\epsilon, -}_i(t) \leq Y_i(t) \leq Y^{\epsilon, +}_i(t)$
    \item $Y^{\epsilon, +}_i(t) = Y^{\epsilon, -}_i(t) + n \epsilon$ for $t \in [T_n^\epsilon, T_{n+1}^\epsilon)$
    \item At time $T_{n}^\epsilon$ all the particles in $Y^{\epsilon, -}$ are in the same position, and similarly for $Y^{\epsilon, +}$ .
\end{itemize}

Assume that we have constructed $T_{n}^\epsilon$; $Y^{\epsilon, +}(t)$ and $Y^{\epsilon, -}(t)$ for all $t \leq T_{n}^\epsilon$. Then since $T_n^{\epsilon}$ is a stopping time, we may use the strong Markov property to note that for $t \geq T_n^\epsilon$, $Y(t)$ behaves as an independent N-BBM. Then since at time $T_n^\epsilon$, $Y^{\epsilon, -}_i(T_n^\epsilon) \leq Y_i(T_n^\epsilon) \leq Y^{\epsilon, +}_i(T_n^\epsilon)$, we may use Lemma \ref{N_coupling} (with the same values of $N$) to couple the processes $Y^{\epsilon, -}(t)$, $Y(t)$ and $Y^{\epsilon, +}(t)$, where we will couple these process up to a stopping time $T_{n+1}^\epsilon$ (about to be defined)

Then in fact if we unpack this coupling process, we are just using the same Brownian Motions for each process between branching times, and immediately after branching times we choose the Brownian motions such that particles in the same relative position have the same Brownian motion (i.e. $Y_i(t)$ evolves by the same Brownian motion as say $Y^{\epsilon, +}_i(t)$) between branching times). It follows then that whilst $Y^{\epsilon, +}$ and $Y^{\epsilon, -}$ are evolving according to this coupling, they are exactly translated versions of each other, since at time $T_n^\epsilon$ both systems had all particles in one place.

Therefore if we define 
$$T_{n+1}^{\epsilon} := \inf \left\{t> T_n + 1: \sup_{1 \leq i,j \leq N} |Y^{\epsilon, -}_i(t) - Y^{\epsilon,-}_j(t)| \leq \epsilon \right\}$$

Then this stopping time will also be the first time that the particles of $Y^{\epsilon, +}$ are all within $\epsilon$ of each other. 

Finally then, at time $T^{\epsilon}_{n+1}$ we define $Y^{\epsilon, -}_1(T^{\epsilon}_{n+1})$ to be the value given by the coupling, and then we redefine $Y^{\epsilon, -}_i(T^{\epsilon}_{n+1}) := Y^{\epsilon, -}_1(T^{\epsilon}_{n+1})$ for all $1 \leq i \leq N$. 

Now let us check that we have satisfied the inductive properties we claimed:

\begin{itemize}
    \item $Y^{\epsilon, -}_i(T^{\epsilon}_{n+1}) := Y^{\epsilon, -}_1(T^{\epsilon}_{n+1}) \leq Y_1(T^{\epsilon}_{n+1}) \leq Y_i(T^{\epsilon}_{n+1})$
    
    Since $Y^{\epsilon, -}_1(T^{\epsilon}_{n+1})$ is the value given by the coupling, and we know by construction that this is bounded above by $Y_1(T^{\epsilon}_{n+1})$

    \item $Y_i(T^{\epsilon}_{n+1}) \leq Y^{\epsilon, +}_i(T^{\epsilon}_{n+1}) \leq Y^{\epsilon, +}_1(T^{\epsilon}_{n+1}) + \epsilon = Y^{\epsilon, -}_1(T^{\epsilon}_{n+1}) + (n+1) \epsilon$
    
    Where the values of $Y^{\epsilon, +}$ above are given by the coupling. These inequalities follow since we know at time $T_{n+1}^\epsilon$ the particles of $Y^{\epsilon, +}$ are at most $\epsilon$ apart, and furthermore, before we define the values at time $T^{\epsilon, +}$ we have that $Y_1^{\epsilon, +}(T^\epsilon_{n+1}) = Y_1^{\epsilon, -}(T^\epsilon_{n+1}) + n \epsilon$
\end{itemize}

Hence, if we now redefine $Y_i^{\epsilon, +}(T^\epsilon_{n+1}) := Y^{\epsilon, -}_1(T^{\epsilon}_{n+1}) + (n+1) \epsilon$ we complete the inductive definition and get the required properties.

Next, note that between times $T^\epsilon_n$ and $T^{\epsilon}_{n+1}$ the processes $Y{\epsilon, \pm}$ evolve as N-BBMs started with all particles at a single point, where here we are invoking the strong Markov property and that $T^\epsilon_n$ are stopping times adapted to the filtration on which $Y$, $Y^{\epsilon, \pm}$ are defined. Hence we can see from the definition of $T_{n}^\epsilon$ that the random variables $S_k^\epsilon := T^\epsilon_k - T^\epsilon_{k-1}$ are i.i.d. And furthermore that the random variables $L_k^\epsilon := Y^{\epsilon, -}(T^\epsilon_k) - Y^{\epsilon, -}(T^\epsilon_{k-1})$ are also i.i.d. Note also that from the coupling we have $Y^{\epsilon, +}(T^\epsilon_k) - Y^{\epsilon, +}(T^\epsilon_{k-1}) = L_k^\epsilon + \epsilon$

The point now, is at the times $T^{\epsilon}_n = \sum_{k=1}^n S_k^\epsilon$, we have that $Y^{\epsilon, -}(T^{\epsilon}_n) = \sum_{k=1}^n L_k^\epsilon$ and $Y^{\epsilon, +}(T^{\epsilon}_n) = \sum_{k=1}^n (L_k^\epsilon + \epsilon)$

It remains to be shown, however that $\mathbb{E}[S_k^\epsilon] < \infty$. To this end let $\mathcal{F}_t$ be the filtration to which the whole process is adapted, and we will show that $S_1^\epsilon = T_1^\epsilon$ is bounded by a geometric random variable. Similar arguments to this will be used several times throughout this dissertation so will not go into too much detail (see proof of Proposition \ref{hitting_time_finite} and for a more detailed version of this argument)

Consider deterministic times $t=k$ for $k = 1, 2, \ldots$. Then at each time let $U_k$ be the event that 
\begin{itemize}
    \item No particles branch during time $t \in [k, k+1/2)$
    \item $\sup_{1 \leq n \leq N-1} \sup_{t \in [k, k+1/2)} |Y_i(t) - Y_i(k)| \leq 1/2$
    \item The particle furthest to the right $Y_N(t)$ moves up by 1 during time $t \in [k, k+1/2)$ so $Y_N(k+1/2) - Y_N(k) > 1$
    \item During time $t \in [k+1/2, k+1)$ None of the Brownian motions driving the particles move by more than $\epsilon / N$
    \item The largest particle branches $N$ times during time $t \in [k+1/2, k+1)$
\end{itemize}

By the Markov property of the system, it is then clear that $U_k$ are independent events of the same probability, and then by standard properties of Brownian motions and exponential distributions it follows that $\mathbb{P}(U_k) > 0$. Additionally if the event $U_k$ happens then $T_1^\epsilon \leq k+1$ since the largest particle branching $N$ times without any particles moving very far ensure that all particles are contained within $\epsilon$ of each other.

Hence 
$$\mathbb{P}(T_1^\epsilon \leq k+1 \; | \; \mathcal{F}_k) \geq \mathbb{P}(U_k) = \mathbb{P}(U_1)$$

And hence by Lemma \ref{stopping_time_expectation} $\mathbb{E}[T^\epsilon_1] < \infty$

Now we wish to show that $\mathbb{E}[L_k^\epsilon] > 0$ for $\epsilon$ sufficiently small. to do this, note that as $n \rightarrow \infty$ we have $T_n^\epsilon \rightarrow \infty$ since by construction $S_k^\epsilon \geq 1$. Hence by Theorem \ref{speed_linear_score}
$$\lim_{n \rightarrow \infty} \frac{Y_1(T^\epsilon_n)}{T^\epsilon_n} = v_N - \mu > 0$$ almost surely

But then by the strong law of large numbers $$\lim_{n \rightarrow \infty} \frac{Y^{\epsilon, -}_1(T^{\epsilon}_n)}{T^\epsilon_n} = \frac{ \lim_{n \rightarrow \infty} \left(\frac{1}{n} \sum_{k=1}^n L_k^\epsilon\right) }{\lim_{n \rightarrow \infty} \left(\frac{1}{n} \sum_{k=1}^n S_k^\epsilon\right)} = \frac{\mathbb{E}[L_1^\epsilon]}{\mathbb{E}[S_1^\epsilon]}$$
and similarly 

$$\lim_{n \rightarrow \infty} \frac{Y^{\epsilon, +}_1(T^{\epsilon}_n)}{T^\epsilon_n}  = \frac{\mathbb{E}[L_1^\epsilon] + \epsilon}{\mathbb{E}[S_1^\epsilon]}$$

And so since $Y_1$ is sandwiched between $Y^{\epsilon, -}$ and $Y^{\epsilon, +}$ it follows that

$$\frac{\mathbb{E}[L_1^\epsilon]}{\mathbb{E}[S_1^\epsilon]} \leq v_N - \mu \leq \frac{\mathbb{E}[L_1^\epsilon] + \epsilon}{\mathbb{E}[S_1^\epsilon]}$$

Then note that the map $\epsilon \mapsto \mathbb{E}[S_1^\epsilon]$ is non-increasing, and hence we make take $\epsilon$ small enough so that $\frac{\epsilon}{\mathbb{E}[S_1^\epsilon]} <\frac{v_N - \mu}{2}$

Hence for some fixed small $\epsilon$ we have that $\mathbb{E}[L_1^\epsilon] \geq \frac{v_N - \mu}{2} > 0$.

Then since $t \geq T^\epsilon_{N_t^\epsilon}$, we calculate:
$$\mathbb{E}[\tau_{R, \mu}] = \mathbb{E}\left[ \int_0^{\infty} \mathbbm{1}_{{\tau_{R, \mu}} \geq t}\right] \leq \mathbb{E}\left[ \int_0^{\infty} \mathbbm{1}_{{\tau_{R, \mu}} \geq T^\epsilon_{N_t^\epsilon}}\right] = \mathbb{E}\left[ \sum_{k=0}^\infty (T^\epsilon_{k+1} - T^\epsilon_k)\mathbbm{1}_{{\tau_{R, \mu}} \geq T^\epsilon_{k}}\right]$$
$$ =  \sum_{k=0}^\infty \mathbb{E}\left[(T^\epsilon_{k+1} - T^\epsilon_k)\mathbbm{1}_{{\tau_{R, \mu}} \geq T^\epsilon_{k}}\right] = \mathbb{E}[S^\epsilon_1]\sum_{k=0}^\infty  \mathbb{E}\left[\mathbbm{1}_{{\tau_{R, \mu}} \geq T^\epsilon_{k}}\right]$$

Where this final step follows since $\tau_{R, \mu}$ and $T^{\epsilon}_n$ are stopping times, so the event $\{\tau_{R, \mu} \geq T^\epsilon_k \} \in \mathcal{F}_{T_k^\epsilon}$ and hence by the strong Markov property, $S^\epsilon_{k+1} = T^\epsilon_{k+1} - T^\epsilon_k$ is independent from $\mathcal{F}_{T_k^\epsilon}$ 

Now proceeding a similar manor to the case of $N=2$, define $Z_n := Y^{\epsilon, -}_1(T_n^\epsilon)$, and 
define $N_t^\epsilon := \sup \{n: T_n^\epsilon \leq t\}$ the counting process associated with the renewal process $T^\epsilon_n$. Then let $\tilde{\tau}_{R, \mu} := \inf\{n: Z_n \geq R\}$ be the first time that $Z_n$ exceeds $R$.

Then $\mathbb{P}(\tau_{R, \mu} \geq T^\epsilon_k) \leq \mathbb{P}(\tilde{\tau}_{R, \mu} \geq k)$ follows from the fact that  $Z_n \leq Y_N(T_n^\epsilon)$.

Hence 
$$\mathbb{E}[\tau_{R, \mu}] \leq \mathbb{E}[S^\epsilon_1]\sum_{k=0}^\infty  \mathbb{E}\left[\mathbbm{1}_{{\tilde{\tau}_{R, \mu}} \geq T^\epsilon_{k}}\right] \leq \mathbb{E}[S^\epsilon_1]\left(1 + \mathbb{E}[\tilde{\tau}_{R, \mu}] \right)$$

Then finally we deduce from Lemma \ref{random_sum_expectation} that $\mathbb{E}[\tilde{\tau}_{R, \mu}] < \alpha' + \beta' R$ for some ($\mu$ dependent) $\alpha'$ and $\beta'$ since $Z_n$ is a random sum of the i.i.d variables $\sum_{k=1}^n L^\epsilon_k$ and $\mathbb{E}[L^\epsilon_k] > 0$

And hence $\mathbb{E}[\tau_{R, \mu}] \leq \alpha + \beta R$
\end{proof}
\end{subsection}
\end{section}

\chapter{Brownian Bees with Drift}

We now move onto Brownian Bees with drift. To provide some (brief) background, the results in this section are inspired by the papers by J. Berestycki et al \cite{swarm_limit, free_boundary_problem}. The results of these papers can be summarized in a commutative diagram, concerning Brownian Bees\footnote{In these papers the authors use the term "N-BBM" to refer to the Brownian Bees, however in this dissertation we reserve this term for the process with killing from the left/right} in $d$ dimensions with no drift (see \cite{swarm_limit} for more discussion and explanation of this diagram):

\begin{figure}[h]
\includegraphics[width=\textwidth]{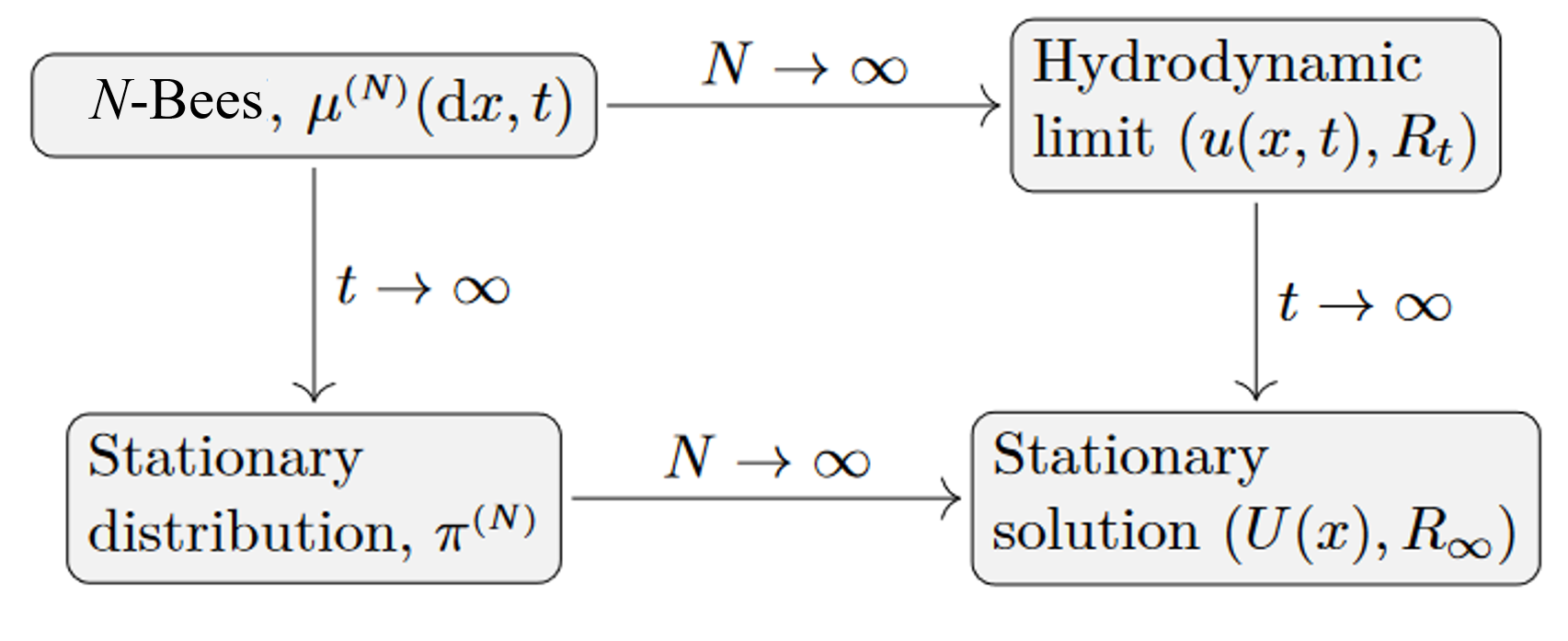}
\end{figure}

For this dissertation we focus on the left hand side of this diagram: the convergence in time of the Brownian Bees process. We state the corresponding result from \cite{swarm_limit}

\begin{theorem}
\label{bees_time_convergence}
    Let $X^{(N)}(t)$ be $d$ dimensional Brownian Bees. Then for $N$ sufficiently large, the process $(X^{(N)}(t), t>0)$ has a unique invariant measure $\pi^{(N)}$, a probability measure on $\left(\mathbb{R}^d \right)^N$. For any $\chi \in \left(\mathbb{R}^d \right)^N$, under $\mathbb{P}_{\chi}$, the law of $(X^{(N)}(t)$ converges in total variation norm to $\pi^{(N)}$ as $t \rightarrow \infty$:
    $$ \lim_{t \rightarrow \infty} \sup_{C} |\mathbb{P}_{\chi} \left(X^{(N)}(t) \in C \right) - \pi^{(N)}(C) | = 0$$
\end{theorem}

This theorem tells us that the Brownian Bees are recurrent and have a unique stationary distribution, to which they converge as $t \rightarrow \infty$. In this section we shall apply a drift $\mu$ to the Brownian Bees, and show that for $\mu$ smaller than the threshold $v_N$ a very similar result to \ref{bees_time_convergence} holds, whereas for $\mu$ larger than this threshold, we have transience, and the system behaves more like an N-BBM.

The main new results of this dissertation are then:

\begin{theorem}
\label{Brownian_Bees_supercritical}
Let $Y^{(N, \mu)} = \left( Y_1(t), \ldots, Y_N(t) \right) \in \mathbb{R}^N$ be a one dimensional Brownian Bee system with drift $\mu$ $\in$ $\mathbb{R}$, ordered such that at each time $t$, $Y_1(t) \leq \ldots \leq Y_N(t)$. Then if $|\mu| > v_N$, where $v_N$ is the asymptotic speed of an N-BBM (\ref{speed_linear_score}), then it holds that:
\begin{equation}
    \lim_{t \rightarrow \infty}\frac{Y_1(t)}{t} = \left( 1 - \frac{v_N}{|\mu|}\right)\mu
\end{equation}
\end{theorem}

Theorem \ref{Brownian_Bees_supercritical} tells us that if the drift $\mu$ is larger than the criticality $v_N$ then the Brownian Bee system is transient, and gives us an asymptotic of its escaping velocity. We contrast this with Theorem \ref{subcritical_stationarity} which tells us the system is recurrent if $|\mu| < v_N$, and converges as $t \rightarrow \infty$:

\begin{theorem}
\label{subcritical_stationarity}
Let $Y^{(N, \mu)} = \left( Y_1(t), \ldots Y_N(t)\right)$ be a one dimensional Brownian Bee system with drift $\mu \in \mathbb{R}$ such that $|\mu| < v_N$ where $v_N$ is asymptotic speed of an N-BBM.
Then there is a unique stationary measure $\pi^{(N, \mu)}$ on $\mathbb{R}^N$ so that for any $\chi \in \mathbb{R}^N$, under $\mathbb{P}_{\chi}$ the law of $Y^{(N, \mu)}$ converges in total variation norm to $\pi^{(N, \mu)}$ as $t \rightarrow \infty$:
$$\lim_{t \rightarrow \infty} \sup_{C} \left| \mathbb{P}_{\chi}\left(Y^{(N, \mu)} \in C \right) - \pi^{(N, \mu)}(C) \right| = 0$$ 
where the supremum is over all Borel measurable $C \in \mathbb{R}^N$

\end{theorem}

A key result in proving this theorem is being able to show that for any initial condition, the particle system will return to 0, and furthermore it will do so in a time of bounded expectation. The method of proof for this is by comparison with an N-BBM, and the use of Theorem \ref{N-BBM_Hitting_Time}, since if all particles are on one side of $0$, the left/right most particles will always be the ones killed, allowing us to couple the Brownian Bees to a N-BBM.

\begin{proposition}
\label{hitting_time_finite}
Let $Y^{(N, \mu)} = \left( Y_1(t), \ldots Y_N(t)\right)$ be a one dimensional Brownian Bee system with drift $\mu \in \mathbb{R}$ such that $|\mu| < v_N$ where $v_N$ is the asymptotic speed of an N-BBM.

Then for any initial condition, for

$$\tau := \inf \left\{t > 0: Y_k(t) = 0 \text{, for some } k \in \{1, \ldots, N\} \right\}$$

It holds that $\tau$ is almost surely finite.
Moreover, if additionally $Y^{(N, \mu)}(0) \in [-R_0, R_0]^N$ has all particles within ${R_0}$ from the origin, then there are deterministic constants $\alpha = \alpha_\mu$ and $\beta = \beta_\mu$ not depending on ${R_0}$ such that 
$$\mathbb{E}[\tau] \leq \alpha + \beta {R_0}$$
\end{proposition}

\begin{section}{Proof of Theorem \ref{Brownian_Bees_supercritical}}

Before stating the proof, we first give an overview of the idea. We first consider the system $\tilde{X}_n =\left(Y_n - \mu t\right)$. Which we can view as a Brunet-Derrida system where the fitness function is changing over time, and we remove the particle that is furthest away from the point $-\mu t$. We then use a lemma that allows us to dominate this process by a "kill from the right process"; which is intuitively saying that "always killing the rightmost particle results in the system as a whole moving further to the left than if we use any other rule". This allows us to show that in the limit the process $\tilde{X}$ always stays to the right of the point $-\mu t$, and hence it becomes a "kill from the right" process.

\begin{lemma} 
\label{coupling_rightmost}
Let $Y^{(N, \mu)}= Y_1, \ldots Y_N$ be a one dimensional Brownian Bee system with drift $\mu \in \mathbb{R}$ and ordered such that $Y_1(t) \leq \ldots \leq Y_N(t)$ for all $t$. Let $X^{(N)} = (X_1, \ldots, X_N)$ be an N-BBM with killing from the right. Then if $X_1, \ldots, X_N$ and $Y_1, \ldots Y_N$ have the same initial condition, there exists a coupling $X^{(N)}$ to $Y^{(N, \mu)}$ such that $X_n(t) \leq Y_n(t) - \mu t$ for $1 \leq n \leq N$ and all time $t$.
\end{lemma}

\begin{proof}
Let $\tilde{X}_n = Y_n(t) - \mu t$ and $\tilde{X}^{(N, \mu)}(t) = \left(\tilde{X}_1, \ldots, \tilde{X}_N \right)$.

A trivial but key observation is that if $\sigma$, $\sigma' : \{1, \ldots, N \} \rightarrow \{1, \ldots N \}$ are any two permutations, then it is sufficient to show that $X_{\sigma(n)}(t) \leq Y_{\sigma'(n)}(t)$ for $1 \leq n \leq N$, and from this we can deduce that $X_n(t) \leq Y_n(t)$.

Now to construct the processes, let $\left( J_i \right)_{i \geq 0}$ be the jump times of a Poisson process of rate $N$ with $J_0 = 0$; $Z_n^i$ be independent Brownian motions and $\left( K_i\right)_{i \geq 1}$ be an independent sequence of uniform random variables on $\{ 1, \ldots N \}$. Then letting $i \geq 1$ we inductively assume that the processes have been constructed up to time $J_{i-1}$.

For $t \in [J_{i-1}, J_i)$, we define 
$$X_n(t) = X_n(J_{i-1}) + Z_n^i(t-J_{i-1})$$
and $$\tilde{X}_n(t) = \tilde{X}_n(J_{i-1}) + Z_n^i(t-J_{i-1})$$

We then relabel $X^{(N)}$ and $\tilde{X}^{(N, \mu)}$ so that they are ordered by position, and then by the remark above and the inductive assumption that $X_n(J_{i-1}) \leq \tilde{X}_n(J_{i-1})$ we deduce that $X_n(t) \leq \tilde{X}_n(t)$ for $t \in [J_{i-1}, J_i)$.

Then at time $J_i$ we delete one particle and duplicate another, with the net result that we are moving one particle to the position of another. We can then try and keep track of which particle moved where, which is a slight shift in perspective as generally we only cared about the total process as opposed to the individual movements of the particles. 

For the process $X^{(N)}$ we want to move the particle $X_N(J_i -)$ to $X_{K_i}(J_i -)$, since we move the furthest right particle. Then for the process $\tilde{X}^{(N, \mu)}$ we move the particle $\tilde{X}_m(J_i -)$ to $\tilde{X}_{K_i}(J_i -)$ where $m$ maximises $|\tilde{X}_m(J_i -) + \mu J_i |$. The idea is that we can then view this move as:
\begin{enumerate}
    \item Move $\tilde{X}_N(J_i -)$ to $\tilde{X}_{K_i}(J_i -)$
    \item Move $\tilde{X}_{N-1}(J_i -)$ to $\tilde{X}_N(J_i -)$
    \item $\ldots$
    \item Move $\tilde{X}_m(J_i -)$ to $\tilde{X}_{m+1}(J_i -)$
\end{enumerate}

\begin{figure}[h]
\includegraphics[width=\textwidth]{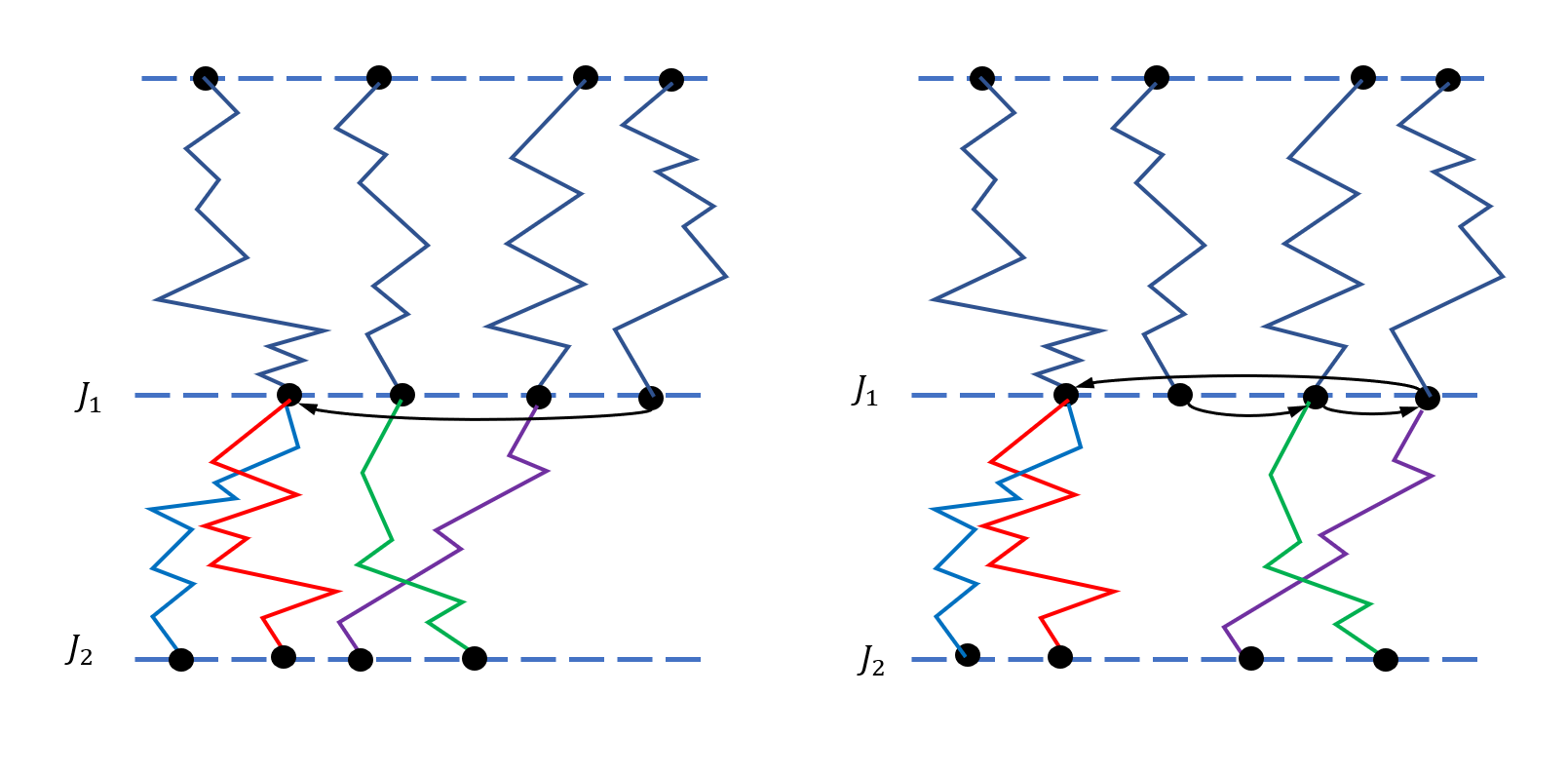}
\caption{\emph{This diagram shows the coupling process. On the left we have the process $X^{(N)}$ and on the right $\tilde{X}^{(N)}$. At time $J_1$ the leftmost particle branches. Then the rightmost particle of $X^{(N)}$ moves to this particle. For $\tilde{X}^{(N)}$, the second particle from the left moves to the leftmost particle, but we view this move as shown by the arrows, and then couple the corresponding Brownian motions so that $\tilde{X}^{(N)}$ lies to the right of $X^{(N)}$ up to time $J_2$}}
\end{figure}

By then assigning each particle at time $t = J_i$ an index according to their order \emph{before} the move, in addition to the ordering $X_1(J_i), \ldots, X_n(J_i)$ then we get permutations $\sigma$ (for $X$) and $\sigma'$ (for $\tilde{X}$). So $\sigma(n)$ is the index of particle just before time $J_i$ before it jumped to $X_n(J_i)$ and similarly for $\sigma'$ and $\tilde{X}_n$ (and with arbitrary choice in case of a tie). We then see that $X_{\sigma(n)}(J_i) \leq \tilde{X}_{\sigma'(n)}(J_i)$, since before the jumps we have $X_n(J_i -) \leq \tilde{X}_n(J_i -)$, for $1 \leq n \leq N$ and:
\begin{itemize}
    \item The particles $\sigma(N)$, $\sigma'(N)$ for $X$ and $\tilde{X}$ have moved to $X_{K_i}(J_i -)$ and $\tilde{X}_{K_i}(J_i -)$ respectively, hence $X_{\sigma(n)}(J_i) = X_{K_i}(J_i -) \leq \tilde{X}_{K_i}(J_i -) = \tilde{X}_{\sigma'(n)}(J_i)$
    \item The particle $\sigma(n)$ for $X$ doesn't move for $1 \leq n \leq N-1$, and so is at position $X_{n}(J_i)$. And the particle $\sigma'(n)$ has either moved down one particle or stayed in the same place for $1 \leq n \leq N-1$, i.e. $\tilde{X}_{\sigma'(n)}(J_i) = \tilde{X}_n(J_i -)$ or $\tilde{X}_{\sigma'(n)}(J_i) = \tilde{X}_{n+1}(J_i -)$. 
    
    However $X_{\sigma(n)}(J_i) = X_n(J_i - ) \leq \tilde{X}_n(J_i -) \leq \tilde{X}_{n+1}(J_i -)$
\end{itemize}

We then see that we have constructed the processes $X^{(N)}$ and $\tilde{X}^{(N, \mu)}$ up to and including time $J_i$, and hence proceeding by induction we have constructed the processes for all time $t$ whilst retaining the monotone property. We can also check that the laws of the processes are correct.
\end{proof}

\begin{proof}[Proof of Theorem \ref{Brownian_Bees_supercritical}:]
We assume without loss of generality that $\mu > 0$ (since the whole system is symmetrical).

Then, let $Y^{(N, \mu)} = \left( Y^{\mu}_1(t), \ldots, Y^{\mu}_N(t) \right) \in  \mathbb{R}^N$ be a Brownian Bee system with drift $\mu$, ordered such that at each time $t$, $Y_1^{\mu}(t) \leq \ldots \leq Y_N^{\mu}(t)$.

Then by Lemma \ref{coupling_rightmost}, we may couple an N-BBM $X^{(N)}$ with with killing from the right to $Y^{(N, \mu)}$ such that $X_n(t) \leq Y_n(t) - \mu t$. As in the proof above, we will denote $\tilde{X}_n = Y_n(t) - \mu t$ and $\tilde{X}^{(N, \mu)}(t) = \left(\tilde{X}_1, \ldots, \tilde{X}_N \right)$

Then, by Theorem \ref{speed_linear_score} we have that 
$$\lim_{t \rightarrow \infty} \frac{X_1(t)}{t} \rightarrow -v_N$$
almost surely,

Hence,
\begin{equation}
\label{tilde_X_lower_bound}
\liminf_{t \rightarrow \infty} \frac{\tilde{X}_1(t)}{t} \geq \lim_{t \rightarrow \infty} \frac{X_1(t)}{t} = - v_N
\end{equation}

The point now, is that the particle that is killed at each reproduction in the process $\tilde{X}^{(N, \mu)}(t)$ is determined by the largest value of $|\tilde{X}_m(t) + \mu t|$, so if we let $T$ be the random time such that for all $t \geq T$,
$$\frac{\tilde{X}_1(t)}{t} \geq -\frac{\mu + v_N}{2}$$
Then by Equation \ref{tilde_X_lower_bound}, $T$ is almost surely finite (since by assumption $\mu > v_N$), and for $t > T$, the killing rule becomes "kill from the right", since for such $t$, $\tilde{X}_m(t) + \mu t$ is always a positive quantity.

So if we fix any $\epsilon > 0$, then we can take a deterministic number $K \in \mathbb{R}$ such that $\mathbb{P}(T \leq K) \geq 1 - \epsilon$. We then construct a process such that it follows the rule of $\tilde{X}$ up to time $K$, and then after that it follows the "kill from the right" rule. Let us call this process $Z^{(N, \mu, K)}(t) = (Z_1(t), \ldots Z_n(t))$, and couple it to $\tilde{X}$ such that we are using the same Brownian motions for both processes. Then for times $t > K$ we can view the process $Z^{(N, \mu, K)}(t)$ as an N-BBM started from the initial conditions of $\left( \tilde{X}_1(K), \ldots, \tilde{X}_N(K) \right)$, and hence by Theorem \ref{speed_linear_score},

\begin{equation}
    \lim_{t \rightarrow \infty} \frac{Z_1(t)}{t} \rightarrow -v_N
\end{equation}
almost surely,

But with probability at least $1 - \epsilon$ we have that $T \leq K$, and then in this case we have that the processes $\tilde{X}^{(N, \mu)}$ and $Z^{(N, \mu, K)}$ are identical. Hence,
\begin{equation}
    \mathbb{P} \left( \lim_{t \rightarrow \infty}\frac{\tilde{X}_1(t)}{t} = \lim_{t \rightarrow \infty}\frac{Z_1(t)}{t} = -v_N \right) \geq 1 - \epsilon
\end{equation}

And as $\epsilon$ is arbitrary, we conclude that $\frac{\tilde{X}_1(t)}{t} \rightarrow -v_N$ a.s. 

Moreover,
$$\frac{Y_1(t)}{t} \rightarrow \mu - v_N$$ a.s.
\end{proof}

\end{section}

\begin{section}{Proof of Theorem \ref{subcritical_stationarity}}
The proof of this theorem is adapted from the proof of Proposition 6.5 and Theorem 1.2 in \cite{swarm_limit}. Throughout the entirety of this section we assume that $Y^{(N, \mu)}$ is our Brownian Bees with drift process and that the drift $\mu$ is such that $|\mu| < v_N$.
\begin{lemma}
\label{recurrent_Harris_chain}
For any $t_1 \in (0, 1)$, the Markov chain $\left(Y^{(N, \mu)}(t_1 n) \right)_{n=0}^{\infty}$ is a positive recurrent strongly aperiodic Harris chain.
\end{lemma}
In an identical way to \cite{swarm_limit} Proposition 6.5 proving Theorem 1.2, Lemma \ref{recurrent_Harris_chain} will be used to prove Theorem \ref{subcritical_stationarity} in combination with Theorems 6.1 and 4.1 of \cite{recurrent_markov_chains_theorems}. These theorems say that a positive recurrent strongly aperiodic Harris chain
admits a unique invariant probability measure, and that the distribution of the state of the
Harris chain after $n$ steps converges to that invariant probability measure as $n \rightarrow \infty$.

\begin{proof}
Let $Z_n := (Y^{(N, \mu)}(t_1 n)$ Then to show that $(Z_n)_{n=1}^{\infty}$ is a positive recurrent strongly aperiodic Harris Chain then by \cite{recurrent_markov_chains_theorems} we only need to show that there exists $\Lambda \subseteq \mathbb{R}^N$ such that:
\begin{enumerate}
    \item $\mathbb{P}_{\chi}(\tau_{\Lambda} < \infty) = 1 \; \forall \chi \in \mathbb{R}^N$, where $\tau_{\Lambda} := \inf \{n \geq 1: Y_n \in \Lambda \}$
    \item There exists $\epsilon > 0$ and a Probability measure $q$ on $\Lambda$ such that $\mathbb{P}_{\chi}(Z_1 \in C) \geq \epsilon q(C)$ for any $\chi \in \Lambda$ and $C \subseteq \Lambda$
    \item $\sup_{\chi \in \Lambda} \mathbb{E}(\tau_\Lambda) < \infty$

\end{enumerate}

We will prove this with the somewhat arbitrary choice of taking $\Lambda = [-1, 1]^N$.

It is now at this point that our proof will differ from that of Proposition 6.5 in \cite{swarm_limit}. The proof of item 2 above and the fact that this Lemma can be used to prove Theorem \ref{subcritical_stationarity} are the same. However for items 1 and 3, \cite{swarm_limit} relied on a comparison for large $N$ between the Brownian Bee system and a system where particles where killed upon reaching a deterministic radius. We however will rely on the results of Proposition  \ref{hitting_time_finite} to prove these points, which has the advantage of also proving the results of \cite{swarm_limit} Theorem 1.2 for small $N$, however the disadvantage that our proof only works for the dimension $d=1$, whereas the results of  \cite{swarm_limit} hold for higher dimensions.

The rough strategy of proof for these three conditions is then

\begin{enumerate}
    \item We use Proposition \ref{hitting_time_finite} to deduce that infinitely often $Y^{(N, \mu)}(t)$ hits 0, and then at each time this happens there is some positive probability of the process branching, and then remaining inside $[-1, 1]^N$ for a time greater than $t_1$, and hence $\tau_{\Lambda} < \infty$
    \item We may condition on no branching events occurring between time $0$ and $t_1$, and then the process moves as independent Brownian Motions, from which we can deduce the result explicitly
    \item We use Proposition \ref{hitting_time_finite} to say that there are times $T_n$ such that the process hits $0$ infinitely often and $T_n \geq T_{n-1}+1$. Then at each time $T_n$ there is a positive probability that the process duplicates so that all particles are in $[-1, 1]^N$ and remains there for time greater than $t_1$. We also need to choose $T_n$ carefully in order to apply the expectation bound from Proposition \ref{hitting_time_finite}, so that at some point between $T_{n}$ and $T_{n+1}$ we are able to control the ball containing the process. To do this, after $T_n$ has occurred we wait until the particle that hit $0$ has duplicated $N$ times, and then argue by means of a coupling and Lemma \ref{radius_bound} that after this time we can control the ball containing the process. Then we choose $T_{n+1}$ to be after this duplication - allowing us to get a bound uniform in $n$ on $\mathbb{E}[T_{n+1} - T_n \; | \; \mathcal{F}_{T_n}]$
\end{enumerate}

We shall now show the second part first:

Conditional on no branching events taking place in time $t \in [0, t_1]$ the particles move as $N$ independent Brownian motions with drift. Hence,
$$\mathbb{P}(Z_1 \in C) \geq e^{- t_1 N} \int_{C} \prod_{i=1}^N \left( \frac{1}{\sqrt{2 \pi t_1}} e^{-\frac{1}{2 t_1} |y_i - \chi_i - \mu t_1|^2} \right) dy_1 \ldots dy_n$$
And then since $\chi, y \in \Lambda = [-1, 1]^N$, we have that $|y_i - \chi_i - \mu t_1|^2 \leq (2 + |\mu| t_1)^2$. Hence letting $\leb{C}$ denote the Lebesgue measure of $C$ in $\mathbb{R}^N$,
$$\mathbb{P}(Z_1 \in C) \geq e^{- t_1 N} \frac{1}{(2 \pi t_1)^{N/2}} \exp \left(-\frac{N(2 + |\mu| t_1)^2}{2 t_1} \right) \leb{C}$$

So taking $q = \leb{C}/\leb{\Lambda}$ and 
$$\epsilon :=  e^{- t_1 N} \frac{1}{(2 \pi t_1)^{N/2}} \exp \left(-\frac{N(2 + |\mu| t_1)^2}{2 t_1} \right) \leb{\Lambda}$$
we have proven item 2 above.

Now to show the first and third point:

We will inductively define our times $T_n$ such that at each $T_n$ there is a positive probability that $\tau_{\Lambda} < T_{n+1}$, and also $\mathbb{E}[T_{n+1} - T_n]$ is bounded in $n$.

Let $\chi \in \mathbb{R}^N$ be the initial position of the particles, and let 
\begin{equation}
\label{R_chi_definition}
    R^{\chi} := \inf\{R \in \mathbb{R}^{>0} : \chi \in [-R, R]^N\}
\end{equation}

Then we define $T_1$ to be the first time that a particle hits 0, so
$$T_1 := \inf \{ t > 0: Y_k(t) = 0 \text{ for some } k \in \{1, \ldots N \}\}$$

Note in particular that by Proposition \ref{hitting_time_finite} we have
\begin{equation}
\label{T_1_bound_equation}
    \mathbb{E}[T_1] < \alpha + \beta R^{\chi}
\end{equation}

Now to inductively define $T_n$, we first define a random variable $S_n$, so that $S_n > T_n + 1$ and the closest particle to 0 has duplicated $N$ times. 

$$S_n := \max \left\{T_n + 1, \inf \{t > T_n: \text{ the particle closest to 0 has branched } N \text{ times by time } t\}\right\} $$
By particle closest to 0, we mean that at a branching time $J_i$, the particle chosen to branch is the closest particle to 0 \emph{at time $J_i$}. 

Note since we have independence of the exponential random variables governing the branching rates from the Brownian motions, $S_n - T_n$ has the distribution of the maximum 1 and the sum of $N$ exponential distributions each of rate 1. i.e.
\begin{equation}
\label{S_n_Gamma}
    S_n - T_n\overset{d}{=} \max\{ \Gamma(N, 1), 1 \}
\end{equation} where $\Gamma(N, 1)$ is a gamma distribution of shape $N$ and rate 1.
The idea is that we can bound the tail of $S_n$, and after time $S_n$, any "far away" particles have been killed and replaced by particles in a controllable distance away from the origin. Next, we define
$$T_n:= \inf \{ t > S_n: Y_k(t) = 0 \text{ for some } k \in \{1, \ldots N \}\}$$
Now we claim that:

\begin{enumerate}[i.]
    \item $R_{S_n} \preccurlyeq L_n$ for $L_n$ i.i.d random variables independent of $\mathcal{F}_{T_n}$ with $\mathbb{E}[L_k] < \infty$, where $R_{S_n}:= \inf\{R \in \mathbb{R}^{>0}:  Y^{(N, \mu)}(S_n) \in [-R, R]^N\}$ and we mean $\preccurlyeq$ in the sense of stochastic domination, i.e. there is a coupling so that $R_{S_n} \leq L_n$ holds for every $n$.
    \item  $\mathbb{P}(\tilde{N} \leq n+1 \; | \; \mathcal{F}_{T_n}) \geq c$ for some constant $c$ where $\tilde{N} := \inf \{n > 0:  t_1 \tau_{\lambda} \leq T_n \}$ has
\end{enumerate}

Suppose for a moment we have proven these two facts. Then by Lemma \ref{stopping_time_expectation} we have that $\mathbb{E}[\tilde{N}] < \infty$.

Then consider, for $n \geq 1$: 
$$\mathbb{E}[T_{n+1} - T_n \; | \; \mathcal{F}_{T_n}]  = \mathbb{E}[S_n - T_n \; | \; \mathcal{F}_{T_n}]+\mathbb{E}[T_{n+1} - S_n \; | \; \mathcal{F}_{T_n}]$$
Note then that $S_k - T_k$ is independent from $\mathcal{F}_{T_n}$ by the strong Markov property, since it only depends on the exponential distributions governing the process. Furthermore, we may then apply Proposition \ref{hitting_time_finite} and use the strong Markov property at time $S_n$ to conclude that

$$ \mathbb{E}[S_n - T_n \; | \; \mathcal{F}_{S_n}] \leq \alpha + \beta R_{S_n}$$

Hence, for $n \geq 1$: 

$$\mathbb{E}[T_{n+1} - T_n \; | \; \mathcal{F}_{T_n}] \leq \mathbb{E}[S_1 - T_1] + \alpha + \beta R_{S_n} \leq \alpha' + \beta \mathbb{E}[R_{S_n} \; | \; \mathcal{F}_{T_n}]$$

Where this is using the fact that $S_k - T_k$ is i.i.d with finite expectation (Since we identified the distribution in Equation \ref{S_n_Gamma})

\begin{equation}
\label{T_n_diff_bound}
     \leq \alpha' + \beta \mathbb{E}[L_n \; | \; \mathcal{F}_{T_n}] =: C
\end{equation}

Where we have used that $R_{S_n} \leq L_n$ almost surely, and that $\left(L_n\right)_{n \geq 1}$ are i.i.d with each $L_n$ independent from $\mathcal{F}_{T_n}$

Additionally, we have that 
\begin{equation}
\label{T_1_bound}
    \mathbb{E}[T_1] \leq \alpha + \beta R^{\chi}
\end{equation}
Where $R_\chi$ is the radius of the ball initially containing all particles (defined in Equation \ref{R_chi_definition})

So if we set $Z_k = T_k - T_{k-1}$ and define
$$M_n := \sum_{k=1}^n\left(Z_k - C- (\alpha + \beta R^{\chi})\right)$$

Then equations \ref{T_n_diff_bound} and \ref{T_1_bound}  tell us that $M_n$ is a supermartingale adapted to the filtration $\mathcal{G}_n := \mathcal{F}_{T_n}$.
We then have the properties:

\begin{itemize}
    \item $\mathbb{E}[|M_{n+1} - M_n| \; | \; \mathcal{F}_{T_n}]$ is almost surely bounded in $n$ (This follows from $\mathbb{E}[Z_n \; | \; \mathcal{F}_{T_n}] \leq C$, and $Z_n$ being positive)
    \item $\tilde{N}$ is a stopping time of the filtration $\mathcal{G}_n$
    \item $\mathbb{E}[\tilde{N}] < \infty$
\end{itemize}

Which is enough to apply Optional Stopping for super-martingales to deduce that:

$$\mathbb{E}[M_{\tilde{N}}] \leq 0$$
And hence
$$\mathbb{E}[t_1 \tau] \leq \mathbb{E}[T_{\tilde{N}}] = \mathbb{E}\left[\sum_{k=1}^{\tilde{N}}Z_k \right] \leq \mathbb{E}\left[\sum_{k=1}^{\tilde{N}}(C + \alpha + \beta R^{\chi})\right] = (C +\alpha + \beta R^{\chi})\mathbb{E}[\tilde{N}]$$

Which in particular after dividing through by $t_1$ proves items 1 and 3 from the start of the proof, since $\mathbb{E}[\tilde{N}] < \infty$ and does not depend on $\chi$; and $R^{\chi} \leq 1$ for $\chi \in \Lambda$

Hence it suffices to prove items i and ii above. Let us start with i.

We consider the positions of the $N$ particles at time $T_n$ and we know that one particle is at the origin at this time. Then we couple the process from time $T_n$ to $N$ free branching Brownian motion processes with drift started at the position of each particle $Y_i(T_n)$ for $1 \leq i \leq N$. Since our model is then branching Brownian motion with selection, the model with a free branching Brownian motion with drift started at every particle will then contain our model, in the sense that the positions of $N$ of the particles in the free BBM model will correspond to $N$ particles in the Brownian Bees with drift model. Let us label these $N$ BBM processes by $X^i_u(t), u \in \mathcal{N}^i_t$, where $1 \leq i \leq N$.  (See Definition \ref{BBM_definition} for the definition of free Branching Brownian motion), so that $\left(X^i_u(t)\right)_{u \in \mathcal{N}^i_t}$ are the $\mathcal{N}^i_t$ particles that are the children of $Y_i(T_n)$, and they are at positions $X^i_u(t)$. Then at time $S_n$, let 

$${\tilde{R}}^i_{S_n} := \inf\left\{R > 0: X^i_u(t) \in [Y^i(T_n) - R, Y^i(T_n) + R], \forall \; T_n \leq t \leq S_n\right\}$$

So $\tilde{R}^i_{S_n}$ gives us a radius surrounding the position of each particle at time $T_n$ that contains all the children of that particle until time $S_n$.

Note then that in this free BBM model, since each particle and its children move independently, then by the strong Markov property we have that ${\tilde{R}}^i_{S_n}$ are i.i.d random variables whose law does not depend on $n$. Furthermore, we may use Lemma \ref{radius_bound} to deduce that:
\begin{equation}
\label{tilde_R^i_bound}
    \mathbb{E}[{\tilde{R}}^i_{S_n}] = C < \infty
\end{equation}
for $C$ independent of $n$.\footnote{To apply lemma \ref{radius_bound} to the drift case, use the fact that (the radius of the process without drift) + $|\mu|(S_n-T_n)$ will dominate the radius of the process with drift. Then equation \ref{S_n_Gamma} tells us that $S_n-T_n$ has the tail of a gamma distribution, which in particular has a tail dominated by an exponential, allowing us to both apply the Lemma, and deduce that $\mathbb{E}[S_n - T_n] < C' < \infty$ for $C'$ independent of $n$}

We claim now that the radius of Brownian Bees at time $S_n$, $R_{S_n}$ is dominated by $2\sum_{i=1}^N \tilde{R}^i_{S_n}$. To see this, suppose for a contradiction that $R_{S_n} \geq 2\sum_{i=1}^N \tilde{R}^i_{S_n}$. Then there must be some particle $k$ so that $Y_k(T_n) + \tilde{R}^k_{S_n} \geq 2\sum_{i=1}^N \tilde{R}^i_{S_n}$, and hence $Y_k(T_n) - \tilde{R}^k_{S_n} \geq 2\sum_{i \neq k} \tilde{R}^i_{S_n}$. But then in the Brownian Bees model there is always at least one alive particle within distance $2\sum_{i \neq k} \tilde{R}^i_{S_n}$ from the origin for $t \in [T_n, S_n])$. 

Since if we denote 
$$D^i := [|Y_{i}(T_n)| - \tilde{R}^{i}_{S_n}, |Y_{i}(T_n)| + \tilde{R}^{i}_{S_n}]$$
Then the distance to the origin of the children of particle $i$ is always in the range $D^i$ for $t \in [T_n, S_n]$. 

Then since there is initially one particle $Y_j(T_n)$ that is at the origin at time $T_n$, then the only way that all the children of this particle can be killed, is if there's some other particle $Y_{j'}(T_n)$ that is closer to the origin at some point, i.e. $D^j$ intersects $D_{j'}$. Then the only way that the particle $j'$ can be killed is if there's another particle that is closer to the origin than $Y_{j'}$ at some point, i.e. $D_{j'}$ intersects $D_{j''}$.

Then continuing this argument by induction, we see that there is always one particle alive within distance $d = 2\sum_{i \neq k} \tilde{R}^i_{S_n}$ from the origin, and furthermore whenever the particle that is closest to the origin branches, all its children remain within distance $d$ from the origin. 

Hence since we know that by time $S_n$ the particle closest to the origin will have branched at least $N$ times, then by time $S_n$, the particle $Y_k(T_n)$ and all its children must have been killed, as every time the particle closest to the origin branches, it adds one more particle which, along with its children, will stay within distance $d$ from the origin. So since all the children of particle $Y_k$ are always further than $d$ away from the origin, these will all be killed.

Hence in conclusion, 
$$R_{S_n} \leq L_n := 2\sum_{i=1}^N \tilde{R}^i_{S_n}$$
which gives the desired bound on $R_{S_n}$ since by the strong Markov property at time $T_n$, we can see that $L_n$ are independent of $\mathcal{F}_{T_n}$, and so i.i.d. And furthermore by equation \ref{tilde_R^i_bound} we have that $\mathbb{E}[L_n] < \infty$, proving item i above.

Now to prove item ii, we will be a little brief with our argument, as similar arguments are used several times throughout this dissertation (see e.g. proof of Proposition \ref{hitting_time_finite} for a more detailed version of this argument)

We argue via the strong Markov property at time $\mathcal{F}_{T_n}$. Let $A_n$ be the event that:

\begin{itemize}
    \item None of the Brownian motions with drift driving the particles will move by more than $1/(2N)$ between time $T_n$ and $T_n + (1-t_1)$
    \item The closest particle to the origin will branch $N$ times between time $T_n$ and $T_n + (1-t_1)$
    \item None of the Brownian motions with drift driving the particles will move by more than $1/2$ between time $T_n + (1-t_1)$ and $T_n + 1$
    \item No branching events will occur between time $T_n + (1-t_1)$ and $T_n + 1$
\end{itemize}

Then by standard properties of Brownian motions and exponential distributions $\mathbb{P}(A_k) > 0$. And furthermore since $T_{n+1} \geq S_n \geq T_n + 1$, by the strong Markov property the events $A_n$ are independent for different $n$, and of the same probability. Furthermore, if the event $A_n$ occurs then $t_1 \tau_{\Lambda}$ will be less than $T_{n+1}$, since the particles will have stayed within $\Lambda$ for a time at least $t_1$, and hence there must be some discrete time-step $k t_1 \in [T_n + (1- t_1), T_n + 1]$ for which the particles lie in $\Lambda$. Hence $\mathbb{P}(\tilde{N} \leq n+1 \; | \; \mathcal{F}_{T_n}) \geq \mathbb{P}(A_1)$ - proving item ii.
\end{proof}

The proof of Theorem \ref{subcritical_stationarity} now follows by the exact same argument to the proof of Theorem 1.2 in \cite{swarm_limit} where we use Lemma \ref{recurrent_Harris_chain} in place of Proposition 6.5.\footnote{Due to limited space, we omit typing out this argument}

\end{section}

\begin{section}{Proof of Proposition \ref{hitting_time_finite}}

In order to prove Proposition \ref{hitting_time_finite}, we start with stating and proving a weaker result: that if all particles are on one side of $0$ initially, then the system will hit 0 in a time of bounded expectation:

\begin{lemma}
\label{hit_if_one_side}
Let $Y^{(N, \mu)} = \left( Y_1(t), \ldots Y_N(t)\right)$ be a one dimensional Brownian Bee system with drift $\mu \in \mathbb{R}$ such that $|\mu| < v_N$ where $v_N$ is the speed of a standard N-BBM. Then if we have initial conditions such that one of:
\begin{itemize}
    \item $Y_1(0) \leq Y_N(0) \leq 0$
    \item or $0 \leq Y_1(0) \leq Y_N(0)$
\end{itemize}
Then the stopping time
$$\tau := \inf \left\{ t > 0 : Y_k(t) = 0 \text{, for some } k \in \{1, \ldots N \} \right\}$$ is almost surely finite

Moreover, if additionally $Y^{(N, \mu)}(0) \in [-R_0, R_0]^N$ has all particles within ${R_0}$ from the origin, then there are deterministic constants $\alpha = \alpha_\mu$ and $\beta = \beta_\mu$ not depending on $R_0$ such that 
$$\mathbb{E}[\tau] \leq \alpha + \beta {R_0}$$ 
\end{lemma}

Having this lemma will then allow us to prove the Proposition, that $\tau$ is finite for \emph{any} initial condition. Furthermore we will also deduce that $\mathbb{E}(\tau)$ is finite.

\begin{proof}[Proof of Lemma \ref{hit_if_one_side}]
    Without loss of generality we may assume that we start in some initial condition with $0 \leq Y_1(0) \leq \ldots \leq Y_N(0)$. Where here we are using the symmetry of the statement to assume we are above 0 as opposed to below 0; which we achieve by potentially reversing the direction of the drift $\mu$ to $-\mu$.
    
    Now let $X^{(N)}$ be a standard N-BBM with killing from the right, and couple $X^{(N)}$ to $Y^{(N, \mu)}$ so that we are using the same random variables to generate each process. Moreover, start $X^{(N)}$ so that it is coupled to the same initial condition as $Y^{(N, \mu)}$. Then while $Y^{(N, \mu)}$ is on the right of $0$, it will behave identically to $\tilde{X}^{(N, \mu)} : = \left(X_1(t) + \mu t, \ldots, X_N(t) + \mu t \right)$. In other words if 
    $$\tau = \inf \left\{ t > 0 : Y_k(t) = 0 \text{, for some } k \in \{1, \ldots N \} \right\}$$
    
    Then for $t < \tau$, $\tilde{X}_n(t)= Y_n(t)$. So $\tau$ is also the first time that the process $\tilde{X}^{(N, \mu)}$ hits 0, and since $\tilde{X}^{(N, \mu)}$ initially starts above 0, it must be the particle $\tilde{X}_1(t)$ that hits 0 first.
    
    However, by Lemma \ref{coupling_rightmost}, we may further couple the process $X^{(N)}$ so that it is dominated by a process $X'^{(N)}$, where initially all particles in $X'^{(N)}$ are at the same position $R_0$. If $\tau'$ is the first time that this system hits $0$, then it follows from this coupling that $\tau < \tau'$ almost surely. However by Theorem \ref{N-BBM_Hitting_Time}, we then deduce that for some $\alpha$, $\beta$ 
    
    $$\mathbb{E}[\tau] \leq \mathbb{E}[\tau'] < \alpha + \beta R_0$$
    
    In particular, $\tau$ is always finite

\end{proof}
    
Now, moving onto the main Proposition:

\begin{proof}[Proof of Proposition \ref{hitting_time_finite}]

Let $\rho := \inf \left\{ t>0 : Y_N(t) <0 \text{ or } Y_1(t) > 0\right\} $, i.e. $\rho$ is the first time that all particles are on one side of $0$. Then conditional on $\rho$ being finite, we may use the strong Markov property at time $\rho$, to get a Brownian Bees system started in position $(Y_1(\rho), \ldots Y_N(\rho))$ - which is necessarily all on one side of $0$. This allows us to apply Lemma \ref{hit_if_one_side} to conclude that $\tau$ is finite. Hence to show the finiteness of $\tau$, it is sufficient to either show that either $\rho < \infty$ or $\tau < \infty$.

Now consider sampling the process $Y^{(N, \mu)}$ at discrete time steps, $t = k$ for $k = 1, 2, \ldots$. Then consider the events $A_k$ and $B_k$ (which the reader may consider to stand for "Above" and "Below"), such that

\begin{itemize}
    \item $A_k$ is the event that at time $k$, if $Y_n(k)$ is the closest particle to $0$, then $Y_n(k) > 0$
     \item $B_k$ is the event that at time $k$, if $Y_n(k)$ is the closest particle to $0$, then $Y_n(k) < 0$
\end{itemize}

Then up to null sets, $A_k$ and $B_k$ are a partition of our probability space $\Omega$. And clearly they are both also $\mathcal{F}_k$ measurable events.

Now we further define events $U_k$ and $D_k$ (which the reader may take as "Up" and "Down")
such that $U_k = U^1_k \cap U^2_k \cap U^3_k \cap U^4_k$ and denoting $\tilde{B}_n(t) = B_n(t) + \mu t$ to be the Brownian motions with drift driving the process $Y_n(t)$ we have:

\begin{itemize}
    \item $U^1_k$ is the event that no branching events occur between time $k$ and time $k+1/2$
    \item $U^2_k$ is the event that if $Y_n(k)$ is the closest particle to $0$ at time $k$, and $\tilde{B}^{n'}(t)$ is the drifting Brownian motion driving the child of $Y_n(k)$ between time $k$ and $k+1/2$, then $\tilde{B}^{n'}(k+1/2) - \tilde{B}^{n'}(k) \geq 1$; and $\left\{ \sup_{1 \leq n \leq N, n \neq n'} \sup_{k \leq t \leq k+1/2} |\tilde{B}^n(t) - \tilde{B}^n(k)| \leq 1/2 \right\}$
    \item $U^3_k$ is the event that $N$ branching events occur between time $k+1/2$ and time $k+1$ and the particles that duplicates is in position $Y_{n'}(t)$ (i.e. there are $n' - 1$ particles to its left) where $Y_{n'}(k)$ is the closest particle to 0 at time $k$
    \item $U^4_k$ is the event $\left\{ \sup_{1 \leq n \leq N} \sup_{k+1/2 \leq t \leq k+1} |\tilde{B}^n(t) - \tilde{B}^n(k)| \leq 1/(2N) \right\}$
\end{itemize}

We may read this event $U_k$ as the particle closest to $0$ at time $k$ moving up by 1 unit, then splitting many times in a row without any Brownian motions moving very far.

\begin{figure}[ht]
\includegraphics[width=\textwidth]{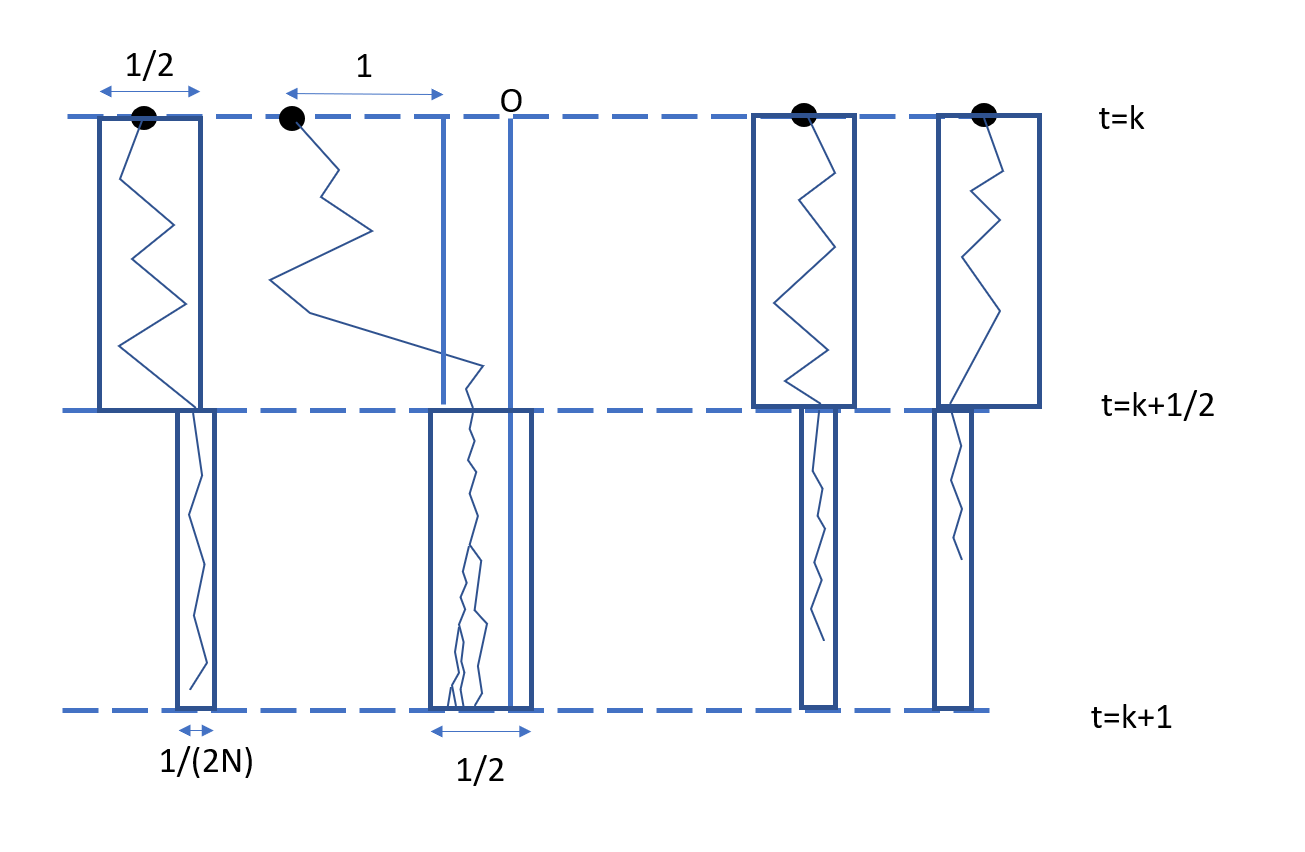}
\caption{\emph{The Event $U_k$. Between $t=k$ and $t=k+1/2$ no particles move more than 1/2, except the closest particle to the origin which moves a distance 1 towards the origin. Then between time $t=k+1/2$ and $t=k+1$ no particles move more than $1/(2N)$. Note due to a possible additive effect with the children of the closest particle branching this only allows us to say that the particles remain with in a radius 1/2 from the closest particle at $t=k+1/2$. This control over the supremum ensures that no other particles can become closer to the origin between $t=k+1/2$ and $t=k+1$. Furthermore the closest particle or its children branch $N$ times in $t \in (k+1/2, k+1)$. The end result is that provided the closest particle at time $k$ was below the origin, then at $t=k+1$ all particles are now either on one side, or one particle has hit 0.}}
\end{figure}

The event $D_k = D^1_k \cap D^2_k \cap D^3_k \cap D^4_k$ is then defined very similarly, except instead of the closest particle moving up, it moves down instead. so $D^1_k = U^1_k$; $D^3_k = U^3_k$; $D^4_k = U^4_k$ and $D^2_k$ is the event $\tilde{B}^{n'}(k+1/2) - \tilde{B}^{n'}(k) \leq -1$ and;
$$\left\{ \sup_{1 \leq n \leq N, n \neq n'} \sup_{k \leq t \leq k+1/2} |\tilde{B}^n(t) - \tilde{B}^n(k)| \leq 1/2 \right\}$$

The point then of constructing these complex events, is that if either $A_k \cap D_k$ or $B_k \cap U_k$ occur, then we have that either $\rho \leq k + 1$ or $\tau \leq k +1$. This is because if the closest particle to $0$ at time $k$ is above 0, and the event $D_k$ occurs, Then at time $k+1/2$, either this particle has crossed 0 and so $\tau \leq k + 1$, or it still lies on the right of 0 and it or its children will be the closest particle to $0$ during time $t \in [k+1/2, k+1]$ - and hence when $N$ branching events occur, all particles will now be on the right of 0 and $\rho \leq k + 1$.

Note also that since we are controlling all the movements of the Brownian motions, then for the event $U^3_k$ or $D^3_k$ the particle in position $n'$, will always be a child of the particle that moved by at least 1.

We now additionally claim that: $(U_k)_{k \geq 1}$ are independent of each other, and similarly for $(D_k)_{k \geq 1}$. And furthermore, $U_k$ and $D_k$ are independent from $\mathcal{F}_k$.

It is then sufficient to show this last point, since $U_k$ is $\mathcal{F}_{k+1}$ measurable, and similarly for $D_k$, so their mutual independence will follow from this.

It is clear then that the events $U^1_k$ and $U^4_k$ are independent from $\mathcal{F}_k$, since we have independence of increments for both the Brownian motions and the exponential distribution governing the next time a particle will branch. It is slightly less clear that $U^2_k$ and $U^3_k$ are independent from $\mathcal{F}_k$, since both events make explicit reference to the "closest" particle at time $t = k$: $Y_{n'}(k)$, and $n'$ is explicitly $\mathcal{F}_k$ measurable. Fortunately though since the random variables $\left(B^1(t) - B^1(k) \right), \ldots, \left(B^N(t) - B^N(k) \right)$ are i.i.d and independent of $\mathcal{F}_k$ for any $t > k$, we can relabel them using some order depending on $\mathcal{F}_k$ and retain independence. This holds similarly for determining which random variable is branching, since we can view the particle in position $n$ as each having their own Poisson process of rate $1$ attached to them. And then since these Poisson processes are i.i.d, we can relabel them by choice using some $\mathcal{F}_k$ dependent order. The details that this relabelling is valid are contained in the appendix Lemma \ref{relabelling_lemma}. Hence from this we deduce that $U_k$ and $D_k$ are independent from $\mathcal{F}_k$.

We now have all the ingredients to our argument, it finally remains to show that $A_k \cap D_k$ or $B_k \cap U_k$ occurs eventually for some $k$, as if this happens almost surely, then it tells us that at least one of $\rho$ or $\tau$ is finite almost surely, from which we deduce the result.

From standard properties of Brownian motions and exponential distributions, we can see that $U_k$ and $D_k$ are both events with positive probability, and furthermore this probability does not depend on $k$; since each event depends only on the jump times between $k$ and $k+1$ and the random variables $\left(\tilde{B}^1(t) - \tilde{B}^1(k) \right), \ldots, \left(\tilde{B}^N(t) - \tilde{B}^N(k) \right)$. Both of which are i.i.d for different values of $k$. Note however it is \emph{not} the case that $\mathbb{P}(U_k) = \mathbb{P}(D_k)$ since the drift $\mu$ gives a higher weight that the drifting Brownian motions $\tilde{B}^n$ will move one way than the other. Let 
\begin{equation}
\label{c_definition}
    c = \min \{\mathbb{P}(U_k), \mathbb{P}(D_k) \}
\end{equation} so $c >0$

Now define a stopping time 
$$\tilde{K} := 1 + \inf \{k \in \mathbb{N}, \omega \in (A_k \cap D_k) \cup (B_k \cap U_k) \}$$
Then by the above discussion we know that almost surely at least one of $\rho$ or $\tau$ is less than $\tilde{K}$. It remains to show that $\tilde{K}$ is finite.

we calculate:
$$\mathbb{P}(\tilde{K} \leq k + 1 \; | \; \mathcal{F}_k) \geq \mathbb{E}\left[ \mathbbm{1}_{(A_k \cap D_k) \cup (B_k \cap U_k)} \; | \; \mathcal{F}_k \right]$$
$$=  \mathbb{E}\left[ \mathbbm{1}_{A_k}\mathbbm{1}_{D_k} + \mathbbm{1}_{B_k}\mathbbm{1}_{U_k} \; | \; \mathcal{F}_k \right] = \mathbbm{1}_{A_k}\mathbb{P}(D_k) + \mathbbm{1}_{B_k}\mathbb{P}(U_k) \geq c$$

Where we have used the fact that $A_k$ and $B_k$ are a partition and $\mathcal{F}_k$ measurable; and that $U_k$ and $D_k$ are independent of $\mathcal{F}_k$. Now we conclude by Lemma \ref{stopping_time_expectation} that $\mathbb{E}[\tilde{K}] < \infty$ and hence $\tilde{K}$ is finite almost surely. Therefore since either $\tau$ or $\rho$ is less than $\tilde{K}$ almost surely, and conditional on $\rho$ being finite we have that $\tau$ is finite, we conclude the first part of the lemma that $\tau$ is finite.

Now for the moreover part we wish to get control over the ball containing the Brownian Bees process at time $\tilde{K}$. Let $R_t$ denote the ball that has contained the process up to time $t$, so 
$$R_t := \sup_{s \leq t} \left\{ \inf_{R > 0} \{R: \; |Y_i(s)| \leq R, \: 1 \leq i \leq N \} \right\}$$

Clearly then $R_t$ is non-decreasing in $t$, and $Y^{(N, \mu)}(t) \subseteq [-R_t, R_t]^N$. The idea now is that we can make a crude approximation where we consider the process without selection allowing us to use Lemma \ref{radius_bound}, which tells us that at time $\tilde{K}$, $R_{\tilde{K}}$ has integrable tails. To make this arguement more precise, we know at time $0$ the particles $Y^{(N, \mu)}(0) \subseteq [-R_0, R_0]^N$.

Then to bound the upper tail and lower tail separately define the quantities:

\begin{itemize}
    \item $R_t^+ := \sup_{s \leq t} \left\{ \inf_{R > 0} \{ Y_i(t) \leq R, \: 1 \leq i \leq N \} \right\}$
    \item $R_t^- := \sup_{s \leq t} \left\{ \inf_{R > 0} \{ Y_i(t) \geq -R, \: 1 \leq i \leq N \} \right\}$
\end{itemize}
Clearly then $R_t$ = $\max\{R_t^+, R_t^-\}$.

Then to bound $R_t^+$, we apply Lemma \ref{coupling_rightmost} to couple the process $Y^{(N, \mu)} - \mu t - R_0$ to an N-BBM with killing from the left, and all $N$ particles starting at the origin \footnote{We are doing a two-step coupling here to dominate the Brownian Bees by a free branching Brownian motion with $N$ particles started at the same position. There are other ways to achieve this but coupling to an N-BBM first to move the particles to the same position makes use of results already proven}. This N-BBM can then be coupled to $N$ free branching Brownian motion processes\footnote{The topic of free branching Brownian motion is somewhat glossed over in this dissertation due to space constraints. The coupling is done in the intuitive way, since N-BBM is a branching Brownian motion with selection, we simply do not perform the selection to make the coupling.} (definition \ref{BBM_definition}) started at the origin, since these correspond to $X^{(N)}$ but without deleting any particles. Then let $R^{i,+}_t$ be the smallest $R$ such that the $i^\text{th}$ free branching Brownian motion process has not hit $R$ by time $t$. It follows then that $$R^+_{\tilde{K}} - \mu \tilde{K} - R_0 \leq \max\{R^{1,+}_{\tilde{K}}, \ldots R^{N,+}_{\tilde{K}} \}$$

Next, $\tilde{K}$ is dominated by an exponential since for $x$ large enough,
\begin{equation}
\label{tilde_K_tail}
\mathbb{P}(\tilde{K} > x) = (1-c)^{\lfloor{x} \rfloor} \leq e^{-\alpha x}
\end{equation}
 where $c$ is from equation \ref{c_definition} and $\alpha$ is some positive constant. Then it follows by Lemma \ref{radius_bound} that $\mathbb{P}(R^{i, +}_{\tilde{K}} \geq x) \leq e^{-c_i \sqrt{x}}$ for $x$ large enough. In particular

$$\mathbb{E} \left[ \max\{R^{1,+}_{\tilde{K}}, \ldots R^{N,+}_{\tilde{K}} \} \right] \leq \sum_{i=1}^N \mathbb{E}[R^{i,+}_{\tilde{K}}] < \infty$$

And note that this quantity is independent of $R_0$.
Hence we conclude that 
\begin{equation}
\label{R_expectation_bound}
 \mathbb{E}[R^+_{\tilde{K}}] \leq R_0 + \mu \mathbb{E}[\tilde{K}] + C = R_0 + C'
\end{equation}

Where $C$, $C'$ are constants depending only on $\mu$. The same argument works for bounding $\mathbb{E}[R^-_{\tilde{K}}]$, so we deduce that $\mathbb{E}[R_{\tilde{K}}] \leq 2R_0 + C''$

Now finally, we know that almost surely one of $\tau$ or $\rho$ is less than $\tilde{K}$, and we know that if $\rho < \tilde{K}$ but $\tau > \tilde{K}$, then at time $\tilde{K}$ all particles must be on one side of $0$.

Hence 
$$\mathbb{E}(\tau) = \mathbb{E}\left[\mathbb{E}(\tau | \mathcal{F}_{\tilde{K}} )\right] \leq \mathbb{E}\left[\tilde{K} + \mathbb{E}\left(\tau \mathbbm{1}_{\tau > \tilde{K}} | \mathcal{F}_{\tilde{K}} \right) \right]$$

Then we can use the strong Markov Property at $\tilde{K}$ and Lemma \ref{hit_if_one_side} to deduce that 
$$ \mathbb{E}(\tau) \leq \mathbb{E}(\tilde{K}) + \mathbb{E}(\tilde{K} + \alpha' + \beta' \tilde{R}_{\tilde{K}}) $$
And finally using equations \ref{tilde_K_tail} and \ref{R_expectation_bound}
$$\mathbb{E}(\tau) \leq \alpha + \beta R_0$$

\end{proof}

\end{section}

\chapter{Conclusion}

We have successfully found the critical case for the recurrence and transience of Brownian Bees. Perhaps the natural next question is what happens at the criticality $\mu = v_N$? In order to answer this question more knowledge would be needed about N-BBMs; for example if we wanted to show that the case $\mu = v_N$ is recurrent we would need a similar result to Theorem \ref{N-BBM_Hitting_Time}, which would require more understanding of the N-BBM than the fairly crude coupling used to prove the above theorem. 

A different way to go further would be to consider the convergence in time of the N-BBM when viewed from the leftmost particle (i.e. the process $\left(X_i(t) - X_1(t)\right)_{i=2}^N$ for $X$ an N-BBM). This has been conjectured in several papers and is widely considered to be true (see e.g. section 8 of \cite{demasi2017hydrodynamics}) - and in fact the method of proof of Theorem \ref{subcritical_stationarity} should allow for the proof of a similar result for N-BBM viewed from the leftmost particle - though this has not been investigated in rigourous detail due to space and time constraints.

\appendix
\chapter{Appendix}
This appendix contains several technical lemmas which are stated here as providing the proofs or statements of them in the section above would detract from the ideas behind the relevant proofs.
\begin{lemma}{Relabelling Lemma}
\label{relabelling_lemma}

Let $X_1, \ldots X_N$ be i.i.d random variables independent of a sigma algebra $\mathcal{G}$. Then if $\rho: \Omega \rightarrow S(N)$ is a random permutation of $\{1, \ldots, N\}$ and $\rho$ is $\mathcal{G}$ measurable, Then $X_{\rho(1)}, \ldots X_{\rho(N)}$ are i.i.d random variables independent of $\mathcal{G}$.

\begin{proof}
$$\mathbb{P}\left(X_{\rho(1)} \leq x_1, \ldots X_{\rho(N)} \leq x_n \right | \:\mathcal{G})$$
$$= \sum_{\sigma \in S(N)} \left( \mathbb{E}\left[ \mathbbm{1}_{X_{\sigma(1)} \leq x_1} \ldots \mathbbm{1}_{X_{\sigma(N)} \leq x_n} \mathbbm{1}_{\rho = \sigma} |\: \mathcal{G}\right] \right)$$
$$= \sum_{\sigma \in S(N)} \left( \mathbb{E}\left[ \mathbbm{1}_{X_\sigma(1) \leq x_1} \ldots \mathbbm{1}_{X_{\sigma(N)} \leq x_n} |\: \mathcal{G}\right] \mathbbm{1}_{\rho = \sigma} \right)$$
$$= \sum_{\sigma \in S(N)} \left( \mathbb{P}\left( X_{\sigma(1)} \leq x_1 \right) \ldots \mathbb{P} \left( X_{\sigma(N)} \leq x_n \right) \mathbbm{1}_{\rho = \sigma} \right)$$
$$= \left( \sum_{\sigma \in S(N)}\mathbbm{1}_{\rho = \sigma} \right) \mathbb{P}\left( X_{1} \leq x_1 \right) \ldots \mathbb{P} \left( X_{N} \leq x_n \right) = \mathbb{P}\left( X_{1} \leq x_1 \right) \ldots \mathbb{P} \left( X_{N} \leq x_n \right)$$

Where we are using the fact that $X_n$ is independent of $\mathcal{G}$ to remove the conditioning, and we are using the fact that $X_1, \ldots X_N$ are i.i.d to replace $\mathbb{P}\left( X_{\sigma(n)} \leq x_1 \right)$ with $\mathbb{P}\left( X_n \leq x_1 \right)$
\end{proof}

\end{lemma}
\begin{lemma}
\label{stopping_time_expectation}
Let $\tau$ be a stopping time of some filtration $\left(\mathcal{G}_n\right)_{n=1}^{\infty}$. Then suppose that there is some $\alpha \in \mathbb{N}$ and $\epsilon > 0$ such that 

$$\mathbb{P}(\tau \leq n + \alpha \; | \; \mathcal{G}_n) \geq \epsilon$$
Then it holds that 
$$\mathbb{P}(\tau > m \alpha) \leq (1- \epsilon)^m$$
And in particular that $\mathbb{E}(\tau) < \infty$

\end{lemma}
\begin{proof}
This is given as exercise E10.5 in \cite{williams_1991}
\end{proof}

\begin{definition}[Free Branching Brownian Motion]
\label{BBM_definition}
A free branching Brownian motion is heuristically a system in which each particle moves independently according a Brownian motion, and at rate $1$, it duplicates into two particles. We write $\mathcal{N}(t)$ for the number of particles alive at time $t$, and $X_u(t), u \in \mathcal{N}(t)$ for the positions of the particles. See \cite{BovierAnton2017Gpot} for a more precise definition and discussion.
\end{definition}






\begin{lemma}[Many-to-one Lemma]
Let $(X_u(t), u \in \mathcal{N}(t))$ be a Branching Brownian Motion process started at $0$. Let $F: C_{[0, T]} \rightarrow \mathbb{R}$ be any measurable function. Then

$$\mathbb{E}\left[ \sum_{u \in \mathcal{N}(t)} F(X_u(s), 0 \leq s \leq t) \right]  = e^t \mathbb{E} \left[ F(B_s, 0 \leq s \leq t)\right]$$
\end{lemma}
\begin{proof}
e.g. \cite{many-to-one} Section 4.1
\end{proof}

\begin{lemma}[Radius Bound]
\label{radius_bound}
Let $(X_u(t), u \in \mathcal{N}(t))$ be a Branching Brownian Motion process started at $0$. Let $T$ be a random variable (possibly dependent) such that for some $\lambda$ and $t$ large enough $$\mathbb{P}(T > t) \leq e^{- \lambda t}$$
Then for 
$$R_t := \sup_{s \leq t} \left\{\inf_{R > 0}\left\{  R : \; |X_u(t)| \leq R, \forall u \in \mathcal{N}(t)\right\} \right\}$$
We have that $\exists c > 0$ so that for $x$ large enough
$$\mathbb{P}(R_T \geq x) \leq e^{-c \sqrt{x}} $$
\end{lemma}

\begin{proof}
We argue first by the many-to-one lemma. 
Let $F(X_s, s \leq t) = \mathbbm{1}_{\sup_{s \leq t} |X_s| \geq x}$ for $x \in \mathbb{R}$
Then,
$$\mathbb{P}(R_t \geq x)  = \mathbb{P}\left(\sup_{u \in \mathcal{N}(t)} \sup_{s \leq t} |X_u(s)| \geq x\right)$$
$$\leq \mathbb{E} \left[ \sum_{u \in \mathcal{N}(t)} \mathbbm{1}_{\sup_{s \leq t} X_u(s) \geq x} \right] = e^t \mathbb{P}\left(\sup_{s \leq t} |B_s| \geq x \right)$$

Then for each $x \in \mathbb{R}$ we choose a deterministic $t_x$ so that $\mathbb{P}(T > t_x) \leq e^{- \lambda \sqrt{x}}$. So for $x$ large enough we can choose $t_x = \sqrt{x}$ by the bound assumed for $T$.

Then 
$$\mathbb{P}\left(R_T \geq x \right) = \mathbb{P}\left(R_T \geq x \;|\; T > t_x \right)\mathbb{P}(T > t_x) +\mathbb{P}\left(R_T \geq x \;|\; T \leq t_x \right)\mathbb{P}(T \leq t_x)$$
$$\leq \mathbb{P}(T > t_x) + \mathbb{P}(R_{t_x} > x)$$
where here we have used the fact that $R_t$ is non-decreasing in $t$.
$$\leq e^{- \lambda \sqrt{x}} + e^{t_x} \mathbb{P}\left(\sup_{s \leq t_x} |B_s| \geq x \right) \leq  e^{- \lambda \sqrt{x}} + 2e^{\sqrt{x}}e^{\frac{-(\sqrt{x})^3}{2}} = e^{- \lambda \sqrt{x}} + 2e^{-x/2}$$
Using the standard bound that $\mathbb{P}(\sup_{s \leq t} B_s \geq \lambda t) \leq e^{\frac{- \lambda^2 t}{2}}$, and the fact that $-B_t$ is a Brownian motion. Hence we can take $c$ so that for $x$ large enough $\mathbb{P}(R_T \geq x) \leq e^{-c\sqrt{x}}$

\end{proof}






\begin{center}
     {\large \textbf{Acknowledgements}}
\end{center}

{\large
With thanks to Julien Berestycki for supervising me during this dissertation}

\addcontentsline{toc}{chapter}{Bibliography}
\bibliography{main}        
\bibliographystyle{plain}  

\end{document}